\documentclass[onefignum,onetabnum]{siamonline250211}
\usepackage{physics}
\usepackage{amsmath}
\usepackage{tikz}
\usepackage{mathdots}
\usepackage{yhmath}
\usepackage{cancel}
\usepackage{color}
\usepackage{siunitx}
\usepackage{array}
\usepackage{multirow}
\usepackage{amssymb}
\usepackage{gensymb}
\usepackage{tabularx}
\usepackage{extarrows}
\usepackage{booktabs}
\usepackage{bbding}
\usetikzlibrary{fadings}
\usetikzlibrary{patterns}
\usetikzlibrary{shadows.blur}
\usetikzlibrary{shapes}
\usepackage[framed,numbered,autolinebreaks,useliterate]{mcode}

\usepackage{lipsum}
\usepackage{amsfonts}
\usepackage{graphicx}
\usepackage{epstopdf}
\usepackage{algorithmic}
\ifpdf
  \DeclareGraphicsExtensions{.eps,.pdf,.png,.jpg}
\else
  \DeclareGraphicsExtensions{.eps}
\fi

\newsiamremark{remark}{Remark}
\newsiamremark{hypothesis}{Hypothesis}
\crefname{hypothesis}{Hypothesis}{Hypotheses}
\newsiamthm{claim}{Claim}

\headers{Generalized Wright equation}{M. D\'iaz, K. Has\'ik,
J. Kopfov\'a,
S. Trofimchuk}

\title{Global stability of Wright-type equations with negative Schwarzian\thanks{ 
This research was supported by the project FONDECYT 1231169 (ANID, Chile) and by the Ministry of Education, Youth and Sports of the Czech Republic (MSMT CR) under RVO funding for I\v CO 47813059.}}

\author{Mauro D\'iaz, 
\thanks{  \email{mauro.gdp@gmail.com}, {Instituto de Matem\'aticas, Universidad de Talca, Campus Lircay, Chile}}
\and Karel Has\'ik, \thanks{
  \email{Karel.Hasik@math.slu.cz}, {Mathematical Institute, Silesian University,  746 01  Opava, Czech Republic}} 
 \and Jana Kopfov\'a, \thanks{
  \email{Jana.Kopfova@math.slu.cz},  {Mathematical Institute, Silesian University,   746 01 Opava, Czech Republic}}
\and Sergei Trofimchuk \thanks{\Envelope \
\email{trofimch@inst-mat.utalca.cl},
{Instituto de Matem\'aticas, Universidad de Talca, Campus Lircay, Chile}} }

\usepackage{bbding}
\usepackage{amsopn}

\ifpdf
\hypersetup{
  pdftitle={An Example Article},
  pdfauthor={D. Doe, P. T. Frank, and J. E. Smith}
}
\fi

\newcommand{\R}{{\Bbb R}}

\newcommand{\N}{{\Bbb N}}

\usepackage{color}

\usepackage{amsmath}
\newtheorem{thm}{Theorem}

\newtheorem{conjecture}[thm]{Conjecture}

\begin{document}

\maketitle
\begin{abstract} Simplicity of the $37/24$-global stability criterion announced by E.M. Wright in 1955 and rigorously proved by B. B\'{a}nhelyi et al in 2014 for  the delayed logistic equation raised the question  of its possible extension for other population models.  
In our study, we answer this question by extending the $37/24$-stability condition for the Wright-type  equations with decreasing smooth nonlinearity $f$ which has  a negative Schwarzian and satisfies the standard negative feedback and boundedness assumptions. The proof contains the  construction and careful analysis of qualitative properties of certain bounding relations.  To validate our conclusions, these relations are evaluated at finite sets of points; for this purpose, we  systematically use interval analysis. 
\end{abstract} 


\section{Introduction} It is well known that the delayed logistic differential equation  

\vspace{-4mm}

\begin{equation}\label{11a}  
y^{\prime}(t)=\alpha y(t)(1-y(t-1)), \quad \alpha >0, \ y \geq 0, 
\end{equation}  
or the Hutchinson equation \cite{H},  is of special importance in population dynamics.  
Surprisingly, the first analytical studies by Edward Maitland Wright \cite{W}, a prominent expert in  number theory \cite{P},  of the equivalent equation 
\begin{equation}\label{11}  
z^{\prime}(t)=-\alpha z(t-1)(1+z(t)), \quad z(t) >-1, 
\end{equation}  
had arisen  from the pure mathematical 
interest of Viscount Cherwell (a scientific advisor to Churchill during the World War II) to find a probabilistic approach for the proof  of the  prime number theorem. See \cite{LRT,MS,WR} for further references and the fascinating history behind this research.   

Equation \eqref{11} is now known  as the Wright's equation. In \cite{W}, E. M. Wright established that for \(\alpha=\ln 2\) all solutions  of this equation  tend to $0$, which, accepting  Cherwell's heuristic reasoning, implies the asymptotic formula of the prime number theorem. Furthermore, Wright also showed that all solutions  of equation \eqref{11} converge to 0 if \(0<\alpha \leq 3 / 2\). He also proved that  for  \(\alpha >\pi / 2\)  this equation has non-convergent slowly oscillating  solutions. Such solutions are defined as  follows:
  
\begin{definition} Solution $z(t)$, $ t \geq T-1,$ is   slowly oscillating if a) its zeros $t_j > T$, (i.e. $z(t_j)=0$), form an increasing  unbounded sequence, b) $t_{j+1}-t_j >1$ for all $j$; c) $z(t)$ is changing its sign at each zero.  
\end{definition}

Regarding the dynamics of \eqref{11} for \(\alpha \in\left(3/2, \pi/{2}\right)\), E. M. Wright wrote, \cite{W}: "My methods, at the cost of considerable elaboration, can be used to extend this result to \(\alpha \leq 37/24\) and, probably, to \(\alpha<1.567\dots\) (compare with \(\pi/2=1.571\dots\)). But the work becomes so heavy for the last step that I have not completed it". 
Considering  these facts, in \cite{W}  the famous 
 Wright conjecture was formulated: \\{\it For \(0<\alpha \leq {\pi}/{2}\), the  solution \(z=0\) of equation \eqref{11} is a global attractor, that is, if \(z(t)>-1\) is a solution of \eqref{11}, then \(z(t) \rightarrow 0\), for \(t \rightarrow \infty\).}

E. M. Wright did not publish his calculations or procedural details leading to the bound $37/24$.
In 2014, B. B\'{a}nhelyi, T. Csendes, T. Krisztin and A. Neumaier essentially contributed   to the solution of the Wright conjecture proving the global stability of the trivial equilibrium of \eqref{11} for \(\alpha \in[1.5,1.5706]\), see \cite{BCKN}. They extended the interval of \(\alpha\) for which \(z=0\) is a global attractor of equation \eqref{11}, following the scheme which we briefly describe below. First, they proved  

\begin{proposition}[Theorem 3.1 in \cite{BCKN}] \label{1.2}  The zero equilibrium is a global attractor of \eqref{11} if and only if equation \eqref{11} does not  have any slowly oscillating periodic solutions.  
  \end{proposition}
  
Next, the change of variables \(x(t)=-\ln (1+z(t))\) transforms  \eqref{11} with \(\alpha>0\) and \(z(t)>-1\) into the equivalent equation
\begin{equation}\label{3}
 x^{\prime}(t)=g(x(t-1)),\quad  \mbox{with}\ g(x)=\alpha\left(e^{-x}-1\right), \alpha>0.
 \end{equation}

In the second step they proved   analytically  that for \(0<\alpha<{\pi}/{2}\), each slowly oscillating periodic solution of equation \eqref{3} has a maximum greater than \(1-{2 \alpha}/{\pi}\) (see Corollary 4.2, \cite{BCKN}). In the third step  they  showed that  for  \(\alpha \in[1.5,1.5706]\),  the maximum of each slowly oscillating periodic solution is less or equal then \(1-{2 \alpha}/{\pi}\) (see Theorem 6.1, \cite{BCKN}). This last part of the proof included a rigorous numerical calculation of high complexity (78 CPU days in \cite{BCKN}). Therefore, combining the three conclusions mentioned above, the result obtained by Wright was improved to \(\alpha \in(0,1.5706]\).  
  
Despite various attempts to prove Wright's conjecture, it remained an open problem for more than 60 years, until  J. B. Van den Berg and J. Jaquette in  \cite{BJ} solved it by invoking  the above mentioned studies in \cite{BCKN}.  It follows from  their work   that  equation \eqref{3} has no slowly oscillating periodic solutions   for $\alpha \in [1.5706, {\pi}/{2}]$, which, together with Proposition 1.2, proves the Wright conjecture, extending the range of \(\alpha\) to the entire interval (0, \(\left. {\pi}/{2}\right]\). 
Once again, rigorous numerical computations played a crucial role in \cite{BJ}. 
  
Now, apart from the particular logistic model \eqref{11a}, there are many other scalar equations of applied interest for which the problem of global stability of equilibria is of primary interest \cite{Kuang}.  As an example, let us mention the following ("food-limiting", cf. \cite{SY}) Michaelis-Menten model
$$
x^{\prime}(t) = -\alpha(1 + x(t)) \frac{x(t - r)}{1 + c\alpha(1 + x(t - r))}, \quad x > -1, \quad c \geq 0, \quad \alpha, r > 0,
$$
which can be easily transformed into the form (\ref{3}) with appropriate $g$. 
Remarkably, the Wright $3/2-$criterion for global stability of the zero equilibrium 
can be extended,  \cite{LPRTT,YW,LTT}, not only for the latter equation, but also  for several other models including  the scalar delay differential equation
\begin{equation}\label{4}
x^{\prime}(t) = f(x(t-1)) 
\end{equation}
with $f$ satisfying the following three conditions  $(\mathbf{H})$: 

\vspace{1mm}

\hspace{-5mm} (H1) (Negative feedback) $f \in C^{3}(\mathbb{R})$ satisfies $x f(x) < 0$ for $x \neq 0$ and $a= f^{\prime}(0) < 0.$  

\vspace{1mm}

\hspace{-5mm}  (H2) $f$ has at most one critical point $x^*$, $f'(x^*)=0$, and there is  $C \in \mathbb{R}$ such that  $f(x) > C$ for all $x \in \mathbb{R}$.

\vspace{1mm}

\hspace{-5mm}  (H3) (Negative Schwarzian)  $S f(x) < 0$ for all $x \neq x^*$, where $S f = f^{\prime\prime\prime} / f^{\prime} - (3 / 2)\left(f^{\prime\prime}\right)^{2} /\left(f^{\prime}\right)^{2} $.  

\vspace{2mm}

Condition  (H1) causes solutions to tend to oscillate  around zero,   while  together with (H2) guarantees the existence of the global compact attractor to \eqref{11}, see \cite{LPRTT}. 

A straightforward verification shows that the nonlinearity $g$ in equation \eqref{3}  meets all the conditions in $(\mathbf{H})$. Therefore the following assertion can be viewed as an extension of the Wright $3/2$-stability theorem. 
\begin{proposition}[Theorem 1.3 in \cite{LPRTT}]\label{1.3} If all conditions in $(\mathbf{H})$ are satisfied and  $f^{\prime}(0) \geq -3 / 2$,  the solution $x = 0$ of \eqref{4} is a global attractor.
\end{proposition}

The proof of Proposition \ref{1.3} is based on the construction of a series of bounds for  the maximum and minimum values, $\mathbf M, \mathbf m$, respectively,  for the  solutions of \eqref{4},  $t \in \R,$ which imply $\mathbf m=\mathbf M=0$.  

In view of  Proposition \ref{1.3}, the following generalization of the Wright's conjecture was also proposed:

\begin{conjecture}[Conjecture 1.2 in \cite{LPRTT}]\label{2} Suppose that the hypotheses  $(\mathbf{H})$ are satisfied. If $-{\pi}/{2} \leq f^{\prime}(0) \leq 0$, then the solution $x = 0$ of \eqref{4} is a global attractor.
\end{conjecture}
Notably, this conjecture holds when $f$ is decreasing and $f''(0)=0$, see Corollary \ref{GS} below (see also the 
results in \cite{KRE}  for a class of odd nonlinearities $f$). 

It is easy to understand how ${\pi}/{2}$ appears in this conjecture. Indeed, the linearization of equation \eqref{4} at the equilibrium $0$ has the form
$
x^{\prime}(t) = f^{\prime}(0) x(t-1). 
$
The complex roots $\lambda_{i}$ of the  associated characteristic equation
$
\lambda - f^{\prime}(0) e^{-\lambda} = 0
$
 define the spectrum  of the zero equilibrium. It can be proved that $\operatorname{Re}\left(\lambda_{i}\right) < 0$ for all $\lambda_{i}$ if and only if $0 < -f^{\prime}(0) < {\pi}/{2}$.

Hence, the above conjecture  simply says that the local stability of the equilibrium $0$ implies its global stability for the scalar delay equation  \eqref{4} of "the Wright type".

 It is worth to mentione that Conjecture \ref{2} has a discrete analogue proposed by H. A. El-Morshedy and E. Liz in \cite{ML}.  They studied the local stability of the unique positive equilibrium $\bar{x}$ of the delay difference equation, known as the Clark equation (if $\alpha=1$)  or the Mackey-Glass type equation (if $\alpha \in (0,1)$)
$$
x_{n+1} = \alpha x_{n} + f\left(x_{n-k}\right), \quad \alpha \in(0, 1],
$$
obtained through Euler discretization of \eqref{4}  \footnote{In parallel, discrete logistic equation $x_{n+1} = \alpha x_n(1-x_{n-k})$ with small delays $k=1,2$ was recently considered by J. Dud\'as and T. Krisztin in \cite{D,DK}. They proved that the local stability of the positive equilibrium $x=1$ implies its global stability in the positive domain.}   
Here $k$ is a fixed integer, a delay, and  $f: \mathbb{R}_{+} \longrightarrow \mathbb{R}_{+}$ is a smooth decreasing function, $f^{\prime}(x) < 0$, with  negative Schwarz derivative. 

It was conjectured in \cite{ML} that the local stability of the positive equilibrium $\bar{x}$ of the Clark equation implies its global stability. Surprisingly, V. L\'opez and E. Parre\~no  showed in   \cite{LP} that this is not true  for some functions $f(x)$ (even rational ones) if $k \geq 2$.
To construct a counterexample,  V. L\'opez and E. Parre\~no in \cite{LP} studied the nature of the Neimark-Sacker bifurcation for the critical parameter 
 $\alpha_{0},$ for which  the equilibrium  $\bar{x}$ loses its local stability and  proved the existence of a subcritical bifurcation at $\bar{x}$ (thus   oscillatory solutions appear for $\alpha < \alpha_{0}$).

The above discussion raises  the question  whether the subcritical Hopf bifurcation at the parameter $f^{\prime}(0) = -{\pi}/{2}$ can occur for \eqref{4}. This information might be very useful in either to 
confirming or refuting Conjecture \ref{2}. The corresponding study was done by I. Bal\'azs and G. R\H{o}st  in  \cite{BR1,BR2}. 
In this respect, Corollary 5.1 in \cite{BR1}, which states that the condition of a negative Schwarzian excludes subcritical Hopf bifurcations at  $f^{\prime}(0)=-{\pi}/{2}$ in  \eqref{4} is highly relevant for us. In particular,  the authors of \cite{BR1} conclude: "We found that it is not possible to construct a counterexample to the conjecture of Liz et al., \cite{GT}, by means of a subcritical Hopf bifurcation". This conclusion justifies further efforts to work on Conjecture \ref{2}. In our paper, we will advance in this direction  by extending the range of $f^{\prime}(0)$ for which $x=0$ is a global attractor of \eqref{4}. More precisely, we will prove the following result:
\begin{theorem}\label{1.4}  Assume $(\mathbf{H})$  and that $f(x)$ is a  decreasing function. If $-{37}/{24} \leq f^{\prime}(0) \leq 0$ then  the solution $x=0$ is a global attractor of \eqref{1.4}.
\end{theorem}
Understanding how $-{37}/{24}$ appears in this conjecture will come from the analysis of the bounding relations, see the computations in Appendix 1. As we will see, the  value  $-37/24$ can be slightly improved within our framework, however, this improvement  is not significant and we omit it.

To prove Theorem \ref{1.4}  we use the following adaptation of Proposition \ref{1.2} in \cite{BCKN}: 

\begin{proposition}\label{1.5} Suppose that the smooth function $f: \mathbb{R} \rightarrow \mathbb{R}$ is decreasing. The solution $x=0$ of \eqref{4} is a global attractor if and only if \eqref{4} has no slowly oscillating periodic solutions.
\end{proposition}
Indeed, a careful analysis of the proof of Theorem 3.1 in \cite{BCKN} shows that it remains valid for smooth decreasing functions that satisfy $(\mathbf{H})$.

The main task in proving Theorem \ref{1.4} is therefore to show the non-existence of slowly oscillating periodic solutions $x(t)$ for  $f^{\prime}(0) \in \left[-{37}/{24},-{3}/{2}\right]$. To achieve  this, we will replace $f(x)$ with a linear rational function $r(x),$ which provides optimal approximation for $f(x)$ at $x=0$. Then,  for each slowly oscillating periodic solution $x(t)$ of \eqref{4}, several appropriate bounds  will be constructed for the maximal value $\mathbf M=\operatorname{max}_{s \in \mathbb{R}} x(s)$ and the minimal value $\mathbf m=\min_{s \in \mathbb{R}} x(s)$  (in the sequel, such a pair $(\bf M, \bf m)$ will be called admissible).  We will  show that, for a certain range of $f^{\prime}(0)$, 
the bounding relations can not be satisfied by any admissible pair of $\mathbf m<0$, $\mathbf M >0$. This part of our proof will use validated (i.e. rigorous) numerics. We show first that  
certain explicit  neighborhood $\mathcal{O}$ of $(0,0) \in \mathbb{R}_{+} \times \mathbb{R}_{-}$ can be removed from our considerations, i.e. $(\mathbf M, \mathbf m) \notin \mathcal{O}$,  in view of the following result.
\begin{proposition}[Theorem 4.1 in \cite{BCKN}] \label{BCKNR} Suppose that a $C^1$-smooth scalar function $f$  satisfies $f(0)=0$ and  
$
-{\pi}/{2}<f^{\prime}(x)<0 \quad \text{for all } x \in(-A, B) \backslash\{0\} 
$
and some $A>0, B>0$. 
Then equation \eqref{4}
does not have slowly oscillating periodic solutions $x(t)$ with $x(\mathbb{R}) \subset (-A, B)$.
\end{proposition}
As we will show at the very beginning of the next section, in Remark \ref{R2.2} and Lemma \ref{LL2.4}, the above proposition has the following important consequences.
\begin{corollary} \label{GS} Assume  that $(\mathbf{H})$ are satisfied  and that $f(x)$ is a  decreasing function with $f''(0)=0$ and $f'(0) \geq  - \pi/2$. Then equation \eqref{4}
does not have slowly oscillating periodic solutions $x(t),$ i.e. 
the trivial equilibrium is a global attractor.  
\end{corollary}
\begin{corollary}  Assume that $(\mathbf{H})$  holds and that $f(x)$ is a  decreasing function with $f''(0)>0$, $f'(0) \geq -37/24$. Then 
the set  $\{(M, m): M \geq 0, m>-0.0093\}$ does not contain admissible values of $(\mathbf M, \mathbf m).$   \end{corollary}
This result allows us in all subsequent discussions to assume that
$m \leq -0.0093.$ We will show later in Section 2, Remark \ref{R2.18}, that  also  $M \geq 0.0094$. Note that throughout this work, we use bold font $(\mathbf M, \mathbf m)$ when referring to the extreme values for a periodic solution $x(t)$.  
The normal fonts $(M,m)$ will be used to denote arguments in the bounding functions. 

In view of the described approach, we organize our work as follows. 

In Section \ref{Sec2}, the lower estimate $\mathbf m \geq L_{f^{\prime}(0)}\left(\mathbf  M, \mathbf m\right)$ for the minimum of a slowly oscillating periodic solution is established. The function $L_{a}(M,m)$ is elementary one, given by an explicit formula (which, however, is rather long).   This estimate, combined with another estimate  from \cite{LPRTT},  defines a narrow region $\mathcal{L}\left(f^{\prime}(0)\right)$ of admissible values $\left(\mathbf  M, \mathbf m\right)$ in the quadrant $(M, m) \in \mathbb{R}_{+} \times \mathbb{R}_{-}$.
We show  that $\mathcal{L}\left(f^{\prime}(0)\right)$ is decreasing with respect to $f^{\prime}(0) \in\left[-{37}/{24},-{3}/{2}\right]$ and that 
\[
\mathcal{L}\left(-{37}/{24}\right) = \{(M, m):  \hat{L}(m) \leq M \leq A(m),\ m \in [-0.25, \sqrt{-{37}/{12\pi}} - 1]\},
\]
where the
continuous functions $A(m), \hat{L}(m)$ for $m \in[-0.25, 0]$, satisfy  $\ \hat{L}^{\prime}\left(0^{-}\right)=-1$ and are strictly decreasing. The domain $\mathcal{L}\left(-{37}/{24}\right)$ is presented in Section \ref{Sec2} on Fig. 7. 
In this section, we use basic interval arithmetic techniques through  INTLAB \cite{SMR} (more precisely, the Matlab/Octave toolbox for Reliable Computing Version 13) to confirm  that all bounding relations are well defined on certain domains specified later. 

Similarly, in Section \ref{Sec3}, the upper bound $\mathbf M \leq\Sigma_{f^{\prime}(0)} \left(\mathbf  M, \mathbf  m\right)$ for the maximum of a slowly oscillating periodic solution is derived. In difference with $L_{a}(M,m)$, 
function $\Sigma_a(M,m)$ is not elementary and its evaluation at the point $(M,m)$ requires to solve  an auxiliary optimization problem. 
We use the INTLAB toolbox to estimate $\Sigma_a(M,m)$ from below and from above, with the necessary precision for each fixed values of the parameters. Validated numerics
is also invoked here to confirm  that $\Sigma_a(M,m)$ is well defined for certain values of $m$ and $M$.  The  upper bound for $\mathbf M$ determines the region $\mathcal{U}\left(f^{\prime}(0)\right)$ containing admissible values $\left(\mathbf  M, \mathbf m\right) \in 
\mathbb{R}_{+} \times \mathbb{R}_{-}$. 
It is analytically demonstrated that $\mathcal{U}\left(f^{\prime}(0)\right)$ is  decreasing with respect to $f^{\prime}(0) \in\left[-{37}/{24},-{3}/{2}\right]$, we will be interested in the largest region $\mathcal{U}\left(-{37}/{24}\right)$. 
\begin{figure}[h]\label{Fig9}
\centering
\scalebox{0.6}
{ \fbox{\includegraphics[width=11cm]{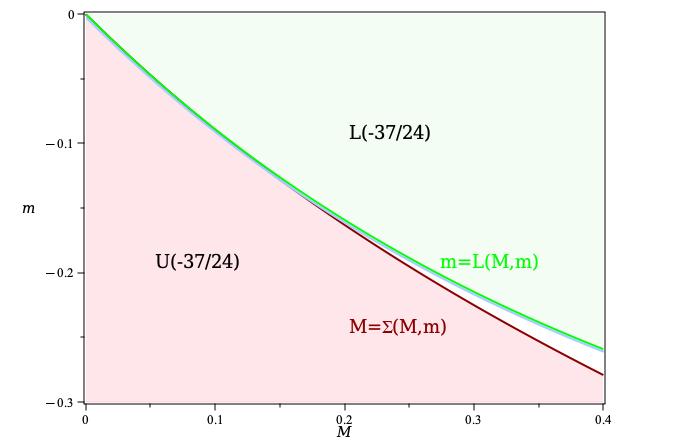}}
 }
\caption{Separation of the regions \( \mathcal{U}(-{37}/{24}) \), \(\mathcal{L}(-{37}/{24})\). }
\end{figure}

\vspace{-0mm}

The above discussion shows that to prove our main result, Theorem \ref{1.4}, it is sufficient to establish that
\begin{equation}\label{8}
\mathcal{U}\left(-{37}/{24}\right) \cap \mathcal{L}\left(-{37}/{24}\right)=\emptyset. 
\end{equation}
In practice, it will be enough to prove that the curves $m= L_{-37/24}(M,m)$ and  $M= \Sigma _{-37/24}(M,m)$ do not intersect in $\frak P:=\{(M, m): M \geq 0.0094, m\leq -0.0093\}$. The main complication here is that these curves are tangent at $(0,0)$ and  the distance between them for the values of $m$ close to $-0.0093$ is of order of $10^{-7}$. See Fig. 1, where 
 both numerically computed curves are shown as boundaries of the red region (containing $\mathcal{U}\left(-{37}/{24}\right)$) and the green region (containing $\mathcal{L}\left(-{37}/{24}\right)$). 
Therefore, by using the interval analysis to rigorously prove (\ref{8}), we need to compute the values of the interval extensions of the functions $L_{-37/24}$ and $\Sigma_{-37/24}$ on the two-dimensional intervals near $(M,m)=(0.0094,-0.0093)$ with a higher precision (e.g. $10^{-9}$ or more). As a consequence, these intervals should have very small size due to the overestimation effects in the interval analysis.   All this requires significant computational resources, which, using a standard office computer, makes it  impossible to check the separation of the mentioned curves on relatively large intervals (even if it can be done on each such very small interval). For a more productive approach, we include $\mathcal{U}\left(-{37}/{24}\right)$ and $ \mathcal{L}\left(-{37}/{24}\right)$  into domains which are bigger but have easily computable boundaries (they will be computed by a fast recursion).  Then, by exploiting the geometric properties of these boundaries,  we show how  evaluations of the functions $L_{-37/24}$ and $\Sigma_{-37/24}$  at recurrently constructed finite set of  24336 points (but not intervals!) allow us  to separate the bigger domains.  
This work is presented in Section \ref{Sec4}  and is essentially based on the use of the INTLAB toolbox for reliable computation \cite{SMR}. Due to its repeated application in the text, the above mentioned trick for  separating two curves is formally described in the next  short section.

\section{Billiards and  verified proofs of the inequalities} \label{SecA} The reader will see this approach used repeatedly in the following sections to prove rigorously 'long' auxiliary inequalities. The underlying idea is quite simple: consider the graphs of two functions $y=p(t)$ and $y=q(t)$, $t \in [t_*,t^*]$, see Fig. 2 (left). Together with two vertical segments on the lines $t=t_*$ and $t=t^*$ they enclose a domain, "a billiard table". 
Then the curves are separated (i.e. $q(t)> p(t)$) if the "billiard ball" entering the table through the left vertical  "wall", after a finite number of non-zero steps and reflections on the borders leaves the table through the right vertical "wall", see Fig. 2 (left).   Indeed, if the curves are not separated, the ball will be stuck inside the domain. 
\begin{lemma} \label{BL} Suppose that the function $q(t)$ is non-decreasing and that the function $p(t)$ is continuous and strictly increasing ($p^{-1}(t)$ will denote its inverse) on some open interval containing $T=[t_*,t^*]$. Then 
$q(t) > p(t)$ for all $t \in T$, if for some $j_0$ the  sequence defined recursively as $t_1=t_*$,  $t_{j+1}=p^{-1}(q(t_j))$, $j=1,.., j_0-1$, is strictly increasing and such that $t_{j_0} > t^*> t_{j_0-1}$.  
\end{lemma}

\begin{figure}[h] \label{F1}

\vspace{0mm}

\centering
{

\tikzset{every picture/.style={line width=0.75pt}} 

\scalebox{0.60}
{
\begin{tikzpicture}[x=0.75pt,y=0.75pt,yscale=-1,xscale=1]

\draw    (193.33,266.67) -- (553.5,267.08) ;
\draw [shift={(555.5,267.08)}, rotate = 180.07] [color={rgb, 255:red, 0; green, 0; blue, 0 }  ][line width=0.75]    (10.93,-3.29) .. controls (6.95,-1.4) and (3.31,-0.3) .. (0,0) .. controls (3.31,0.3) and (6.95,1.4) .. (10.93,3.29)   ;
\draw    (238.5,266.5) -- (238.5,24.5) ;
\draw [shift={(238.5,22.5)}, rotate = 90] [color={rgb, 255:red, 0; green, 0; blue, 0 }  ][line width=0.75]    (10.93,-3.29) .. controls (6.95,-1.4) and (3.31,-0.3) .. (0,0) .. controls (3.31,0.3) and (6.95,1.4) .. (10.93,3.29)   ;
\draw  [dash pattern={on 0.84pt off 2.51pt}]  (396.5,147.08) -- (450,146.58) ;
\draw [shift={(428.25,146.79)}, rotate = 179.46] [fill={rgb, 255:red, 0; green, 0; blue, 0 }  ][line width=0.08]  [draw opacity=0] (8.93,-4.29) -- (0,0) -- (8.93,4.29) -- cycle    ;
\draw  [dash pattern={on 0.84pt off 2.51pt}]  (488,71.92) -- (487,109.92) ;
\draw [shift={(487.67,84.42)}, rotate = 91.51] [fill={rgb, 255:red, 0; green, 0; blue, 0 }  ][line width=0.08]  [draw opacity=0] (8.93,-4.29) -- (0,0) -- (8.93,4.29) -- cycle    ;
\draw  [dash pattern={on 0.84pt off 2.51pt}]  (284.5,183) -- (396.5,186.08) ;
\draw [shift={(345.5,184.68)}, rotate = 181.58] [fill={rgb, 255:red, 0; green, 0; blue, 0 }  ][line width=0.08]  [draw opacity=0] (8.93,-4.29) -- (0,0) -- (8.93,4.29) -- cycle    ;
\draw  [dash pattern={on 0.84pt off 2.51pt}]  (396.5,147.08) -- (396.5,186.08) ;
\draw [shift={(396.5,160.08)}, rotate = 90] [fill={rgb, 255:red, 0; green, 0; blue, 0 }  ][line width=0.08]  [draw opacity=0] (8.93,-4.29) -- (0,0) -- (8.93,4.29) -- cycle    ;
\draw  [dash pattern={on 0.84pt off 2.51pt}]  (283.5,267.08) -- (283,188.5) ;
\draw [shift={(283.22,222.79)}, rotate = 89.64] [fill={rgb, 255:red, 0; green, 0; blue, 0 }  ][line width=0.08]  [draw opacity=0] (8.93,-4.29) -- (0,0) -- (8.93,4.29) -- cycle    ;
\draw [color={rgb, 255:red, 74; green, 144; blue, 226 }  ,draw opacity=1 ][line width=2.25]    (309.48,207.48) .. controls (310.01,207.49) and (249,217.25) .. (311.49,207.51) .. controls (373.97,197.76) and (426,193.42) .. (512,77.25) ;
\draw [color={rgb, 255:red, 240; green, 13; blue, 13 }  ,draw opacity=1 ][line width=2.25]    (284.5,183) .. controls (425.5,155.92) and (480.95,81.83) .. (497.23,62.87) .. controls (513.5,43.92) and (505.54,52.1) .. (509.5,47.25) ;
\draw  [dash pattern={on 0.84pt off 2.51pt}]  (452.5,112.08) -- (450.5,144.08) ;
\draw [shift={(451.91,121.6)}, rotate = 93.58] [fill={rgb, 255:red, 0; green, 0; blue, 0 }  ][line width=0.08]  [draw opacity=0] (8.93,-4.29) -- (0,0) -- (8.93,4.29) -- cycle    ;
\draw  [dash pattern={on 0.84pt off 2.51pt}]  (454.5,107.58) -- (484.5,108.08) ;
\draw [shift={(474.5,107.92)}, rotate = 180.95] [fill={rgb, 255:red, 0; green, 0; blue, 0 }  ][line width=0.08]  [draw opacity=0] (8.93,-4.29) -- (0,0) -- (8.93,4.29) -- cycle    ;
\draw  [dash pattern={on 4.5pt off 4.5pt}]  (396.5,186.08) -- (396.5,266.58) ;
\draw  [dash pattern={on 4.5pt off 4.5pt}]  (452.5,148.58) -- (452,266.58) ;
\draw  [dash pattern={on 4.5pt off 4.5pt}]  (488.5,117.08) -- (487.5,264.08) ;
\draw  [dash pattern={on 4.5pt off 4.5pt}]  (507,84.25) -- (506.5,276.75) ;
\draw  [dash pattern={on 0.84pt off 2.51pt}]  (483.5,75.08) -- (522.5,74.42) ;
\draw [shift={(508,74.66)}, rotate = 179.02] [fill={rgb, 255:red, 0; green, 0; blue, 0 }  ][line width=0.08]  [draw opacity=0] (8.93,-4.29) -- (0,0) -- (8.93,4.29) -- cycle    ;
\draw  [dash pattern={on 4.5pt off 4.5pt}]  (520,70.58) -- (520,265.42) ;

\draw (276,178.23) node [anchor=north west][inner sep=0.75pt]    {\Large $\bullet $};
\draw (402,178.37) node [anchor=north west][inner sep=0.75pt]  [rotate=-85.14]  {\Large $\bullet $};
\draw (389.5,140) node [anchor=north west][inner sep=0.75pt]    {\Large $\bullet $};
\draw (480,69.73) node [anchor=north west][inner sep=0.75pt]    {\Large $\bullet $};
\draw (445,138.23) node [anchor=north west][inner sep=0.75pt]    {\Large $\bullet $};
\draw (514.5,69.23) node [anchor=north west][inner sep=0.75pt]    {\large $\bullet $};
\draw (446,101.23) node [anchor=north west][inner sep=0.75pt]    {\Large $\bullet $};
\draw (480.5,101.98) node [anchor=north west][inner sep=0.75pt]    {\Large $\bullet $};
\draw (278,260.73) node [anchor=north west][inner sep=0.75pt]    {\Large $\bullet $};
\draw (209.17,22.73) node [anchor=north west][inner sep=0.75pt]    {$y$};
\draw (549.67,272.23) node [anchor=north west][inner sep=0.75pt]    {$t$};
\draw (258.17,271.73) node [anchor=north west][inner sep=0.75pt]    {$t_{*} =t_{1}$};
\draw (374.67,244.23) node [anchor=north west][inner sep=0.75pt]    {$t_{2}$};
\draw (428.67,243.23) node [anchor=north west][inner sep=0.75pt]    {$t_{3}$};
\draw (465.67,243.23) node [anchor=north west][inner sep=0.75pt]    {$t_{4}$};
\draw (501.67,272.23) node [anchor=north west][inner sep=0.75pt]    {$t^{*}$};
\draw (409.17,55.23) node [anchor=north west][inner sep=0.75pt]  [color={rgb, 255:red, 239; green, 27; blue, 27 }  ,opacity=1 ]  {$y=q( t)$};
\draw (303.17,215.23) node [anchor=north west][inner sep=0.75pt]  [color={rgb, 255:red, 74; green, 144; blue, 226 }  ,opacity=1 ]  {$y=p( t)$};
\draw (524.17,244.23) node [anchor=north west][inner sep=0.75pt]    {$t_{5}$};

\end{tikzpicture}

}
{\scalebox{0.60}
{\begin{tikzpicture}[x=0.75pt,y=0.75pt,yscale=-1,xscale=1]

\draw    (193.33,266.67) -- (553.5,267.08) ;
\draw [shift={(555.5,267.08)}, rotate = 180.07] [color={rgb, 255:red, 0; green, 0; blue, 0 }  ][line width=0.75]    (10.93,-3.29) .. controls (6.95,-1.4) and (3.31,-0.3) .. (0,0) .. controls (3.31,0.3) and (6.95,1.4) .. (10.93,3.29)   ;
\draw    (238.5,266.5) -- (238.5,24.5) ;
\draw [shift={(238.5,22.5)}, rotate = 90] [color={rgb, 255:red, 0; green, 0; blue, 0 }  ][line width=0.75]    (10.93,-3.29) .. controls (6.95,-1.4) and (3.31,-0.3) .. (0,0) .. controls (3.31,0.3) and (6.95,1.4) .. (10.93,3.29)   ;
\draw  [dash pattern={on 0.84pt off 2.51pt}]  (353.5,174.58) -- (399.5,173.75) ;
\draw [shift={(381.5,174.08)}, rotate = 178.96] [fill={rgb, 255:red, 0; green, 0; blue, 0 }  ][line width=0.08]  [draw opacity=0] (8.93,-4.29) -- (0,0) -- (8.93,4.29) -- cycle    ;
\draw  [dash pattern={on 0.84pt off 2.51pt}]  (432,133.75) -- (432,152.75) ;
\draw [shift={(432,136.75)}, rotate = 90] [fill={rgb, 255:red, 0; green, 0; blue, 0 }  ][line width=0.08]  [draw opacity=0] (8.93,-4.29) -- (0,0) -- (8.93,4.29) -- cycle    ;
\draw  [dash pattern={on 0.84pt off 2.51pt}]  (280.5,195) -- (348.5,196.25) ;
\draw [shift={(319.5,195.72)}, rotate = 181.05] [fill={rgb, 255:red, 0; green, 0; blue, 0 }  ][line width=0.08]  [draw opacity=0] (8.93,-4.29) -- (0,0) -- (8.93,4.29) -- cycle    ;
\draw  [dash pattern={on 0.84pt off 2.51pt}]  (351,175.25) -- (351,197.25) ;
\draw [shift={(351,179.75)}, rotate = 90] [fill={rgb, 255:red, 0; green, 0; blue, 0 }  ][line width=0.08]  [draw opacity=0] (8.93,-4.29) -- (0,0) -- (8.93,4.29) -- cycle    ;
\draw  [dash pattern={on 0.84pt off 2.51pt}]  (283.5,267.08) -- (283,188.5) ;
\draw [shift={(283.22,222.79)}, rotate = 89.64] [fill={rgb, 255:red, 0; green, 0; blue, 0 }  ][line width=0.08]  [draw opacity=0] (8.93,-4.29) -- (0,0) -- (8.93,4.29) -- cycle    ;
\draw [color={rgb, 255:red, 74; green, 144; blue, 226 }  ,draw opacity=1 ][line width=2.25]    (310.48,207.48) .. controls (311.01,207.49) and (250,217.25) .. (312.49,207.51) .. controls (374.97,197.76) and (427,193.42) .. (513,77.25) ;
\draw [color={rgb, 255:red, 240; green, 13; blue, 13 }  ,draw opacity=1 ][line width=2.25]    (284.5,183) .. controls (425.5,155.92) and (480.95,81.83) .. (497.23,62.87) .. controls (513.5,43.92) and (505.54,52.1) .. (509.5,47.25) ;
\draw  [dash pattern={on 0.84pt off 2.51pt}]  (398,156.42) -- (397,174.25) ;
\draw [shift={(397.86,158.84)}, rotate = 93.21] [fill={rgb, 255:red, 0; green, 0; blue, 0 }  ][line width=0.08]  [draw opacity=0] (8.93,-4.29) -- (0,0) -- (8.93,4.29) -- cycle    ;
\draw  [dash pattern={on 0.84pt off 2.51pt}]  (408,153.58) -- (428.5,154.25) ;
\draw [shift={(423.25,154.08)}, rotate = 181.86] [fill={rgb, 255:red, 0; green, 0; blue, 0 }  ][line width=0.08]  [draw opacity=0] (8.93,-4.29) -- (0,0) -- (8.93,4.29) -- cycle    ;
\draw  [dash pattern={on 4.5pt off 4.5pt}]  (351.5,194.75) -- (351,263.08) ;
\draw  [dash pattern={on 4.5pt off 4.5pt}]  (397,174.25) -- (396.5,264.75) ;
\draw  [dash pattern={on 4.5pt off 4.5pt}]  (432,152.75) -- (430,266.75) ;
\draw  [dash pattern={on 4.5pt off 4.5pt}]  (459,112.75) -- (458.5,265.75) ;
\draw  [dash pattern={on 0.84pt off 2.51pt}]  (491,76.58) -- (510,74.75) ;
\draw [shift={(505.48,75.19)}, rotate = 174.49] [fill={rgb, 255:red, 0; green, 0; blue, 0 }  ][line width=0.08]  [draw opacity=0] (8.93,-4.29) -- (0,0) -- (8.93,4.29) -- cycle    ;
\draw  [dash pattern={on 4.5pt off 4.5pt}]  (499.5,64.08) -- (500,266.25) ;
\draw  [dash pattern={on 0.84pt off 2.51pt}]  (433.5,131.58) -- (458.5,130.75) ;
\draw [shift={(451,131)}, rotate = 178.09] [fill={rgb, 255:red, 0; green, 0; blue, 0 }  ][line width=0.08]  [draw opacity=0] (8.93,-4.29) -- (0,0) -- (8.93,4.29) -- cycle    ;
\draw  [dash pattern={on 0.84pt off 2.51pt}]  (461.5,112.08) -- (479,111.75) ;
\draw [shift={(475.25,111.82)}, rotate = 178.91] [fill={rgb, 255:red, 0; green, 0; blue, 0 }  ][line width=0.08]  [draw opacity=0] (8.93,-4.29) -- (0,0) -- (8.93,4.29) -- cycle    ;
\draw  [dash pattern={on 0.84pt off 2.51pt}]  (473.5,94.58) -- (492,94.25) ;
\draw [shift={(487.75,94.33)}, rotate = 178.97] [fill={rgb, 255:red, 0; green, 0; blue, 0 }  ][line width=0.08]  [draw opacity=0] (8.93,-4.29) -- (0,0) -- (8.93,4.29) -- cycle    ;
\draw  [dash pattern={on 4.5pt off 4.5pt}]  (477.5,100.58) -- (477.5,266.25) ;
\draw  [dash pattern={on 4.5pt off 4.5pt}]  (491,76.58) -- (491.5,266.75) ;
\draw  [dash pattern={on 4.5pt off 4.5pt}]  (507,77.42) -- (508.5,272.42) ;
\draw  [dash pattern={on 0.84pt off 2.51pt}]  (458.5,119.25) -- (458.5,130.75) ;
\draw [shift={(458.5,118.5)}, rotate = 90] [fill={rgb, 255:red, 0; green, 0; blue, 0 }  ][line width=0.08]  [draw opacity=0] (8.93,-4.29) -- (0,0) -- (8.93,4.29) -- cycle    ;
\draw  [dash pattern={on 0.84pt off 2.51pt}]  (492,94.25) -- (492,78.08) ;
\draw [shift={(492,81.17)}, rotate = 90] [fill={rgb, 255:red, 0; green, 0; blue, 0 }  ][line width=0.08]  [draw opacity=0] (8.93,-4.29) -- (0,0) -- (8.93,4.29) -- cycle    ;
\draw  [dash pattern={on 0.84pt off 2.51pt}]  (477,112.25) -- (477.5,100.58) ;
\draw [shift={(477.46,101.42)}, rotate = 92.45] [fill={rgb, 255:red, 0; green, 0; blue, 0 }  ][line width=0.08]  [draw opacity=0] (8.93,-4.29) -- (0,0) -- (8.93,4.29) -- cycle    ;

\draw (277,189) node [anchor=north west][inner sep=0.75pt]    {\Large $\bullet $};
\draw (356,189) node [anchor=north west][inner sep=0.75pt]  [rotate=-85.14]  {\Large $\bullet $};
\draw (346,168) node [anchor=north west][inner sep=0.75pt]    {\Large $\bullet $};
\draw (426,127) node [anchor=north west][inner sep=0.75pt]    {\Large $\bullet $};
\draw (391,168) node [anchor=north west][inner sep=0.75pt]    {\Large $\bullet $};
\draw (501,69) node [anchor=north west][inner sep=0.75pt]    {\Large $\bullet $};
\draw (391,148) node [anchor=north west][inner sep=0.75pt]    {\Large $\bullet $};
\draw (426,148) node [anchor=north west][inner sep=0.75pt]    {\Large $\bullet $};
\draw (278,261) node [anchor=north west][inner sep=0.75pt]    {\Large $\bullet $};
\draw (209.17,22.73) node [anchor=north west][inner sep=0.75pt]    {$y$};
\draw (549.67,272.23) node [anchor=north west][inner sep=0.75pt]    {$t$};
\draw (258.17,271.73) node [anchor=north west][inner sep=0.75pt]    {$t_{*} =t'_{1}$};
\draw (328.17,242.23) node [anchor=north west][inner sep=0.75pt]    {$t'_{2}$};
\draw (373.17,241.73) node [anchor=north west][inner sep=0.75pt]    {$t'_{3}$};
\draw (412.17,243.73) node [anchor=north west][inner sep=0.75pt]    {$t'_{4}$};
\draw (493.5,270.48) node [anchor=north west][inner sep=0.75pt]    {$t^{*}$};
\draw (409.17,55.23) node [anchor=north west][inner sep=0.75pt]  [color={rgb, 255:red, 239; green, 27; blue, 27 }  ,opacity=1 ]  {$y=q( t)$};
\draw (287.67,216.9) node [anchor=north west][inner sep=0.75pt]  [color={rgb, 255:red, 74; green, 144; blue, 226 }  ,opacity=1 ]  {$y=p( t)$};
\draw (441.17,243.23) node [anchor=north west][inner sep=0.75pt]    {$t'_{5}$};
\draw (485,70) node [anchor=north west][inner sep=0.75pt]    {\Large $\bullet $};
\draw (452,127) node [anchor=north west][inner sep=0.75pt]    {\Large $\bullet $};
\draw (452,106) node [anchor=north west][inner sep=0.75pt]    {\Large $\bullet $};
\draw (471,106) node [anchor=north west][inner sep=0.75pt]    {\Large $\bullet $};
\draw (471,89) node [anchor=north west][inner sep=0.75pt]    {\Large $\bullet $};
\draw (484,89) node [anchor=north west][inner sep=0.75pt]    {\Large $\bullet $};
\draw (461.67,244.23) node [anchor=north west][inner sep=0.75pt]    {$t'_{6}$};
\draw (510.17,242.23) node [anchor=north west][inner sep=0.75pt]    {$t'_{8}$};
\draw (478.17,244.23) node [anchor=north west][inner sep=0.75pt]    {$t'_{7}$};

\end{tikzpicture}}

\caption{On the left: the ball running across the billiard table and rebounding off its edges; $j_0=5$. On the right: real computation  algorithm  uses lower approximations of $t_j$ and $q(t_j)$; $J_0=8$.}}}
\end{figure}
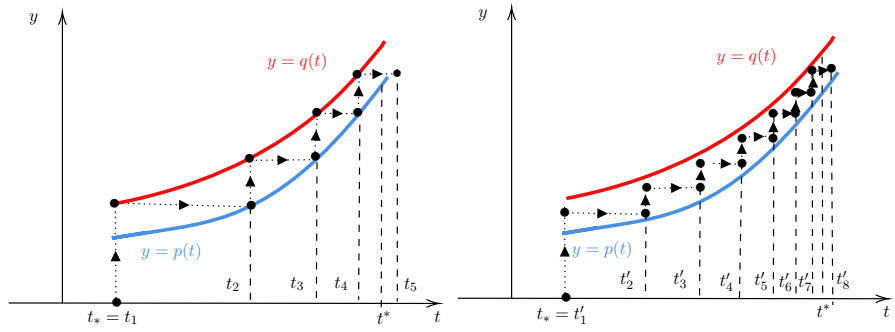 
\begin{proof} Indeed, since $t_2 = p^{-1}(q(t_1)) >t_1$, $t_3 = p^{-1}(q(t_2)) >t_2$, we find that $p(t_2) =q(t_1)>p(t_1)$ and $p(t_2) < q(t_2)$. Since $p(t) < p(t_2) = q(t_1) \leq q(t)$ for all $t \in [t_1, t_2)$, the required inequality is proved for 
$t \in [t_1, t_2]$. Clearly, the same argument works for each closed sub-interval $[t_k,t_{k+1}]$  of $[t_*,t^*]$ and for the half-open interval $[t_{j_0-1},t_{j_0})$.   
Since the union of these sub-intervals contains $[t_*,t^*]$, the statement of the lemma follows. 
\end{proof}
The advantage of applying Lemma \ref{BL} is clear: we can prove  complicated inequalities of the form $q(t)> p(t)$ by establishing  the monotonicity and continuity  properties of the functions  $q(t), p(t)$ and after computing a finite number of terms $t_j$.  Note that each inequality $r(t)>0$ with continuous locally non-constant function $r(t)$ of bounded variation can be presented in such a form (the Jordan decomposition). We assume that, by using the interval analysis and computers, the computation of $t_j$ can be done with the necessary precision. 
In practice, to obtain the rigorous estimates, in the construction of the sequence $t_j$ we have to round down (or approximate from below)  $q(t_j)$ and $t_j$ as shown on Fig. 2 (right). This shortens the lengths  $t'_{j+1}-t'_{j}$ of "steps" and therefore the modified sequence $\{t'_j\}$ might have a bigger number $J_0$ of terms than $\{t_j\}$ has, i.e. $j_0 \leq J_0$. Therefore, in contrast to $j_0$, the number $J_0$ depends on the particular implementation of the approximation procedure. In our work, while applying Lemma \ref{BL}, we will indicate the functions playing the roles of $p, q$ and the number $J_0$ of used iterations provided by the respective MATLAB script in the appendices. 
\section{ A bound for the minimum of a slowly oscillating periodic solution} \label{Sec2}

Fix the real numbers  $a<0,\ b>0$ and consider the following set of scalar functions 
$$K_{a,b}= \{f \in C^3(\R):  f \text{ satisfies } (H1), (H2), f'(0) = a, f''(0) = 2b, $$$$(Sf)(x) \leq 0 \text{ for all } x \text{ such that } f'(x) \neq 0 \}$$
together with the linear rational function 
$r(x,a,b) = {a^2x}/{(a-bx)}$ for $x \in (ab^{-1}, \infty).$

\begin{lemma}\cite[Lemma 2.1]{LPRTT} \label{2.1}
For each $f \in K_{a,b}$ such that $(Sf)(x) < 0$ whenever $f'(x) \neq 0,$ it holds that  \( f'(x) > r'(x) \) for \( x > ab^{-1} \) and 
$$r(x,a,b) <f(x) \quad \mbox{for} \ x>0; \qquad r(x,a,b) >f(x) \quad \mbox{for} \ x \in (ab^{-1}, 0).$$
\end{lemma}

\begin{remark} \label{R2.2} Consider  equation (\ref{4}) with $f \in K_{a,b}$. 

If  $f''(0)=0$ then  Lemma \ref{2.1}  shows that $\min_{x \in \R} f'(x) =f'(0)$ and Proposition \ref{BCKNR} assures that there are no slowly oscillating periodic solutions. This proves Corollary \ref{GS}. Thus, the  case  $f''(0)=0$ can be excluded from our subsequent analysis.

Applying the change of variables $x=ky$ in \eqref{4} with a parameter $k>0$, we obtain the  equation \begin{equation}\label{4a}
y^{\prime}(t) = g(y(t-1)), \ \mbox{with} \ g(y): = k^{-1}f(ky). 
\end{equation}
Then  $g'(0)=f'(0)$, \ $g''(0)= kf''(0)$. 
The dynamical properties of equations (\ref{4}) and (\ref{4a}) are identical,  if $k>0,$ i.e.  
$g(y)$ satisfies the same hypothesis $(\mathbf{H})$ as $f(x)$ does.

If $k<0$, 
$(\mathbf{H2})$ holds for (\ref{4a}) only if $g(y) = k^{-1}f(ky)$ is bounded from below, i.e.  $f(x)$ is bounded from above. But this fact is known from \cite[Corollary 2.2]{LPRTT}. Note that $g''(0)=kf''(0) >0$ in this case. 

{\bf Thus, without loss of generality,  we assume in the sequel that $f''(0)> 0$.  }

 This allows us, after choosing $k=-2f'(0)/f''(0)>0$,  to assume that $a= f'(0)=-f''(0)/2=-b$. We can therefore use a simpler comparison function 
$$r(x) =r(x,a,-a) = {a x}/(1+x)$$
with  $a=f'(0) <0$ for $x \in (-1, \infty)$.

\end{remark}


\begin{figure}[h] \label{F1}

\vspace{-40mm}

\centering{
\scalebox{0.60}
{

\tikzset{every picture/.style={line width=0.75pt}} 

\begin{tikzpicture}[x=0.75pt,y=0.75pt,yscale=-1,xscale=1]

\draw    (64.5,153.5) -- (636.5,156.49) ;
\draw [shift={(638.5,156.5)}, rotate = 180.3] [color={rgb, 255:red, 0; green, 0; blue, 0 }  ][line width=0.75]    (10.93,-3.29) .. controls (6.95,-1.4) and (3.31,-0.3) .. (0,0) .. controls (3.31,0.3) and (6.95,1.4) .. (10.93,3.29)   ;
\draw [line width=2.25]    (66.5,184.5) .. controls (257.5,-211.5) and (425.5,458.5) .. (588.5,132.5) ;
\draw  [dash pattern={on 4.5pt off 4.5pt}]  (201.5,54.5) -- (200.5,153.5) ;
\draw    (487.5,224.5) -- (488.5,155.5) ;
\draw    (367.5,266.5) -- (367.5,24.5) ;
\draw [shift={(367.5,22.5)}, rotate = 90] [color={rgb, 255:red, 0; green, 0; blue, 0 }  ][line width=0.75]    (10.93,-3.29) .. controls (6.95,-1.4) and (3.31,-0.3) .. (0,0) .. controls (3.31,0.3) and (6.95,1.4) .. (10.93,3.29)   ;
\draw  [dash pattern={on 4.5pt off 4.5pt}]  (201.5,54.5) -- (367.5,56.5) ;
\draw [color={rgb, 255:red, 208; green, 2; blue, 27 }  ,draw opacity=1 ][line width=2.25]    (251.5,55.5) -- (366.5,155.5) ;
\draw  [dash pattern={on 4.5pt off 4.5pt}]  (368.5,224.5) -- (487.5,224.5) ;
\draw [color={rgb, 255:red, 237; green, 18; blue, 18 }  ,draw opacity=1 ][line width=2.25]    (55.5,54.5) -- (251.5,55.5) ;

\draw (374,131.4) node [anchor=north west][inner sep=0.75pt]    {$0$};
\draw (482,128.4) node [anchor=north west][inner sep=0.75pt]    {$1$};
\draw (555,101.4) node [anchor=north west][inner sep=0.75pt]    {$x=x( t)$};
\draw (631,127.4) node [anchor=north west][inner sep=0.75pt]    {$t$};
\draw (190,177.4) node [anchor=north west][inner sep=0.75pt]    {$t_{m}{}_{a}{}_{x}$};
\draw (374,49.4) node [anchor=north west][inner sep=0.75pt]    {$ \mathbf M$};
\draw (290,71.4) node [anchor=north west][inner sep=0.75pt]  [color={rgb, 255:red, 218; green, 13; blue, 13 }  ,opacity=1 ]  {$x=r(\mathbf   M) t$};
\draw (339,219.4) node [anchor=north west][inner sep=0.75pt]    {$\mathbf  m$};

\end{tikzpicture}

}
}
\vspace{-33mm}

\caption{\hspace{0cm} Slowly oscillating periodic solution $x(t)$, $x(t_{max})= \mathbf  M,$ of  \eqref{4} and its upper bounds. }
\end{figure}
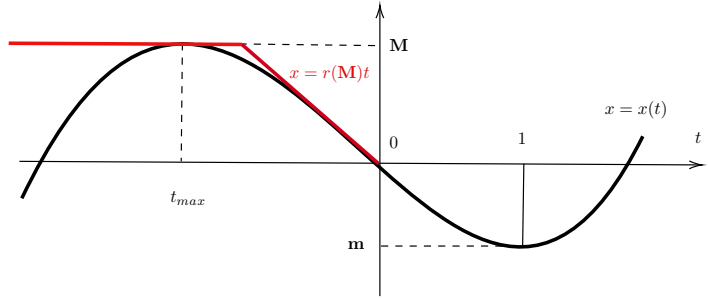

Suppose now that $x(t)$ is a slowly oscillating periodic  solution of  \eqref{4}. Then, according to \cite[Theorem 7.1]{MPS},
$x(t)$ has exactly one maximum and one minimum per period.  We will use the following notation (see Fig. 3):  $$0> \mathbf m:= \min_{\R}x(t) \leq x(t) \leq \max_{\R}x(t)=x(t_{max})=: \mathbf  M >0.$$

\vspace{0mm}

We can assume without loss of generality that the minimum is achieved at the point $t=1,$ i.e.
$x(1)= \mathbf  m,$ which implies that $x'(1) = 0  $ and, consequently, 
$f(x(0))= 0$.
Assumption (H1) then implies that 
$x(0) = 0.$
Integrating equation \eqref{4} from $t=0$ to $t=1$ we get 
\begin{equation}\label{m}
\mathbf  m=x(1) =x(0) + \int_0^1 f(x(s-1))d= \int_{-1}^0 f(x(s))ds.
\end{equation}
\begin{lemma}  Suppose that \( a \in (-2, 0) \) and $f''(0) >0$.  Then \( \mathbf  m > -1 \).
\end{lemma}
\begin{proof} According to the proof of Corollary 3.3 in \cite{LPRTT}, for \( a \in (-2, 0) \), we have
\[
m > r\left(-\frac{r(\infty)}{2}\right) = r\left(\frac{-a}{2}\right) = \frac{-a^2}{2 - a} > -1.
\]
\end{proof}
\begin{lemma}\label{LL2.4}
 The strip $\{(M, m): M \geq 0, m\geq -0.0093\}$ does not contain admissible values of $(\mathbf  M, \mathbf  m)$ whenever $a=f'(0) \geq -37/24$.  More precisely, all such values of \( (\mathbf  M, \mathbf  m) \) lie in the region
\[
\left\{ (M, m) \in \mathbb{R}_+ \times \mathbb{R}_- : m \leq \sqrt{-{2a}/{\pi}} - 1 \right\}.
\]
\end{lemma}
\begin{proof} Lemma \ref{2.1} shows that \( r'(x) \leq f'(x) \) for \( x > -1 \). Now, \( -\pi/2 < r'(x) \) for \( x > -1 \) is equivalent to the inequality
$
x > \sqrt{-2a/\pi} - 1. 
$
That is, \( -\pi/2 < r'(x) < f'(x) < 0 \) for all 
$$ x \in  (-0.00931\dots, + \infty) = \left( \sqrt{37/12\pi} - 1, +\infty \right) \subset  \left( \sqrt{-2a/\pi} - 1, +\infty \right).$$ Therefore, by Proposition \ref{BCKNR}, the equation \eqref{4} 
has no slowly oscillating periodic solutions \( x(t) \) with \( x(\R) \subset \left(-0.0093, +\infty \right) \) for all $a= f'(0) \in [-37/24,0)$.
\end{proof}

Next,  Lemma \ref{2.1},  assumption (H1) and  $f(0)=r(0)=0$ imply that for all $t \in \R$
\begin{eqnarray}
f(x(t-1)) \geq \begin{cases}
0, \quad &\text{ if } x(t-1)) \leq 0,\\
r(x(t-1)), \quad &\text{ if } x(t-1)) > 0.
\end{cases}
\end{eqnarray}
Thus, using the monotonicity of $r$, we obtain
$$x'(t) = f(x(t-1)) \geq r(\mathbf  M) $$ for all $t \in \R.$
Since $x(0)=0$, it follows that
$$-x(t) = \int_t^0 x'(s) ds \geq -r(\mathbf  M) t, \quad \mbox{i.e.} \quad 
x(t) \leq r(\mathbf  M)t \quad \mbox{for all} \ \  t
 \leq 0.$$

This implies that 
\begin{equation}\label{13}
x(t) \leq \min \{\mathbf  M, r(\mathbf  M)t\}, \quad t \leq 0, 
\end{equation}
see Fig. 3.  This estimate was used in  previous works, especially in \cite{LPRTT}. To improve it, besides the estimate for \( x'(t) \), we will derive an appropriate lower bound for \( x''(t) \) on the interval \( (t_{\max}, 0) \). Probably,   this idea was used for the first time   by E. M. Wright in his computations leading to the number $37/24$. Here we will estimate  \( x''(t) \) in  the following fashion. First, observe that  differentiating equation \eqref{4} leads to
\begin{equation}\label{d}
x''(t) = f'(x(t - 1)) \cdot x'(t - 1) = f'(x(t - 1)) \cdot f(x(t - 2)).
\end{equation}
From Lemma \ref{2.1},  hypotheses {\bf (H)} and the properties of the function $r(t)$ we know that $$ f(x(t - 2)) \leq \max_{u \in [\mathbf  m, \mathbf  M]} f(u) < r(\mathbf  m).$$ By the same lemma, for all $x > 0$ it holds  that 
\begin{equation}\label{dd}
f'(x) > \frac{a}{(1 + x)^2} > a = f'(0), 
\text{ and } \min_{x \geq 0} f'(x) = f'(0) = a.
\end{equation}
Therefore, since \( f \) is decreasing on \( \mathbb{R} \), and  \( x(t - 1) > 0 \) if \( t \in (t_{\max}, 0) \), we have from \eqref{d} and \eqref{dd} that
\begin{equation}\label{14}
x''(t) \geq a  \max_{u \in [\mathbf  m, \mathbf  M]} f(u) \geq a  r(\mathbf  m), \quad t \in [t_{\max}, 0]. 
\end{equation}
Let now $\alpha_+ <\beta_+$ be real numbers  and  \( p_+(t) \) be a quadratic polynomial such that
\[
p_+(\alpha_+) = \mathbf  M, \quad p'_+(\alpha_+) = 0, \quad p'_+(\beta_+) = r(\mathbf  M), \quad p''_+(t) = a r(\mathbf  m) \ \mbox{for all} \ t\in \R.
\]
\vspace{0mm}
\begin{figure}[h] \label{F2}

\centering

\scalebox{0.75}
{

\tikzset{every picture/.style={line width=0.75pt}} 

\begin{tikzpicture}[x=0.75pt,y=0.75pt,yscale=-1,xscale=1]

\draw    (151.5,186.5) -- (578.5,189.98) ;
\draw [shift={(580.5,190)}, rotate = 180.47] [color={rgb, 255:red, 0; green, 0; blue, 0 }  ][line width=0.75]    (10.93,-3.29) .. controls (6.95,-1.4) and (3.31,-0.3) .. (0,0) .. controls (3.31,0.3) and (6.95,1.4) .. (10.93,3.29)   ;
\draw    (479.5,262.5) -- (480.99,17) ;
\draw [shift={(481,15)}, rotate = 90.35] [color={rgb, 255:red, 0; green, 0; blue, 0 }  ][line width=0.75]    (10.93,-3.29) .. controls (6.95,-1.4) and (3.31,-0.3) .. (0,0) .. controls (3.31,0.3) and (6.95,1.4) .. (10.93,3.29)   ;
\draw [line width=1.5]    (377.5,66.5) -- (499.5,212.5) ;
\draw [line width=1.5]    (155.5,92.5) -- (481.5,93.5) ;
\draw [color={rgb, 255:red, 215; green, 19; blue, 19 }  ,draw opacity=1 ][line width=1.5]    (200.5,254.5) .. controls (328.5,40.5) and (387.5,39.5) .. (516.5,251.5) ;
\draw  [dash pattern={on 4.5pt off 4.5pt}]  (361,94) -- (358.5,189.5) ;
\draw  [dash pattern={on 4.5pt off 4.5pt}]  (445,150) -- (444.5,189.5) ;
\draw    (315.5,187.5) ;

\draw (493,86.4) node [anchor=north west][inner sep=0.75pt]    {$\mathbf M$};
\draw (381,41.4) node [anchor=north west][inner sep=0.75pt]    {$x=r(\mathbf M) t$};
\draw (565,165.4) node [anchor=north west][inner sep=0.75pt]    {$t$};
\draw (493,17.4) node [anchor=north west][inner sep=0.75pt]    {$x$};
\draw (346,192.4) node [anchor=north west][inner sep=0.75pt]    {$\alpha _{+}$};
\draw (435,192.4) node [anchor=north west][inner sep=0.75pt]    {$\beta _{+}$};
\draw (490,162.4) node [anchor=north west][inner sep=0.75pt]    {$0$};
\draw (219,232.4) node [anchor=north west][inner sep=0.75pt]  [color={rgb, 255:red, 241; green, 29; blue, 29 }  ,opacity=1 ]  {$x=p_{+}( t)$};
\draw (300,164.4) node [anchor=north west][inner sep=0.75pt]    {};
\draw (307,183.4) node [anchor=north west][inner sep=0.75pt]    {};

\end{tikzpicture}

}
\vspace{0mm}

\caption{\hspace{0cm} Graph of the quadratic polynomial $p_+(t)$.}
\end{figure}
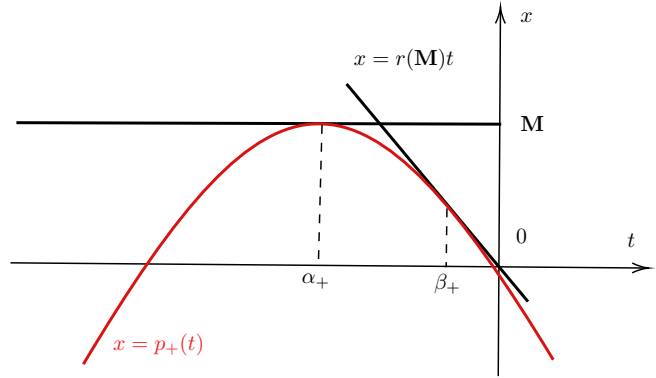 
The Taylor expansion of $p_+(t)$ around \( \alpha_+ \) is therefore
\[
p_+(t) = \mathbf M + \frac{a r(\mathbf  m)}{2}(t - \alpha_+)^2. 
\]
Since  \( p'_+(\beta_+) = r(\mathbf  M)\), we find that
$$
 a r(\mathbf  m)(\beta_+ - \alpha_+) = r(\mathbf  M),  \quad \mbox{so that} \quad \beta_{+} - \alpha_{+} = \frac{r(\mathbf  M)}{a  r(\mathbf  m)}.$$
Also \( p_+(\beta_+) = r(\mathbf  M)\beta_+ \), so that 
\[
 \mathbf   M + \frac{a r(\mathbf  m)}{2}(\beta_+ - \alpha_+)^2 =r(\mathbf  M)\beta_+. 
\]
Thus
\begin{equation}\label{20}
\beta_{+}=\frac{\mathbf  M}{r(\mathbf  M)} +\frac{r(\mathbf  M)}{2ar(\mathbf  m)}, \qquad
 \alpha_{+} =\frac{\mathbf  M}{r(\mathbf  M)} -\frac{r(\mathbf  M)}{2ar(\mathbf  m)}.
\end{equation}
The following lemma justifies the position of the point $\beta_{+}$ shown on Fig. 4.  
\begin{lemma}\label{L2.4}
$\beta_{+} < 0$  if $a \in \left (-2, 0 \right)$.  
\end{lemma}
\begin{proof}
Observe that $\beta_+ \geq 0$ is equivalent to 
\begin{equation*}\label{mM} 
 \frac{\mathbf  m}{1+\mathbf  m} \geq -\frac{\mathbf  M}{2(1 + \mathbf  M)^2}.   
\end{equation*}
Consequently,
\begin{equation}\label{mMd} 
\mathbf  m \geq -\frac{\mathbf  M}{2(1+\mathbf  M)^2 + \mathbf  M} > -\frac{\mathbf  M}{2 + 5\mathbf  M} > -\frac{1}{5}, \quad \mbox{so that} \quad 
\mathbf M > -\frac{2\mathbf  m}{1 + 5\mathbf  m}.
\end{equation}
On the other hand, according to Lemma 3.4 in \cite{LPRTT}, we have:
\begin{equation} \label{Am}
\mathbf  M <  \mathbf  m + r(\mathbf  m) + \frac{1}{r(\mathbf  m)} \int_\mathbf  m^0 r(s) \, ds =: A_a(\mathbf  m)     
\end{equation}
\[
= \mathbf  m + \frac{a\mathbf  m}{1 + \mathbf  m} + \frac{1 + \mathbf  m}{\mathbf  m} \int_\mathbf  m^0 \frac{s}{1+s} \, ds
\]
\[
 \leq \mathbf  m - 2 \cdot \frac{\mathbf  m}{1+\mathbf  m} - 1 - \mathbf  m + \frac{1+\mathbf  m}{\mathbf  m} \ln(1+\mathbf  m)
\]
\begin{equation}\label{23}
= -1 -2 \cdot \frac{\mathbf  m}{1+\mathbf  m} + \frac{1+\mathbf  m}{\mathbf  m} \ln(1+\mathbf  m). 
\end{equation}
Thus, 
\[
-\frac{2\mathbf  m}{1+5\mathbf  m} < -1 - 2 \cdot \frac{\mathbf  m}{1+\mathbf  m} + \frac{1+\mathbf  m}{\mathbf  m} \ln(1+\mathbf  m)
\]
for  $\mathbf m \in \left(-{1}/{5}, 0\right)$. This yields a contradiction, since 
\[ \zeta(m):=
-1 - 2 \cdot \frac{m}{1+m} + \frac{1+m}{m} \ln(1+m) +\frac{2\cdot m}{1+5m} <0 
\]
because  \ $\zeta(0^-)=0$ and $\zeta'(m) >0$  for all $m \in \left(-{1}/{5}, 0\right)$.
\end{proof}
\begin{remark}\label{R2.5}
In Lemma \ref{L2.4}, we assume that the pair $(\mathbf  m, \mathbf  M)$ represents the extremal values of a slowly oscillating periodic solution $x(t)$. On the other hand, we can also consider $\alpha_+$ and $\beta_+$ as functions of two independent variables $m, M$ regardless of the solution $x(t)$. Then the conclusion of the lemma, $\beta_+ <0$, holds
in the domain $0< M\leq  A_a(m)$ and also when $m < -1/9$.  For the latter estimate, note that $-M/(2(1+M)^2+M) \geq  -1/9$ for $M >0$ and use (\ref{mMd}).    
\end{remark}
Next, let us consider the piecewise continuous function \( z_+(t) \), see Fig. 4:
\begin{eqnarray}\label{312}
z_+(t) = 
\begin{cases} 
M, & \text{if } t \leq \alpha_+, \\
p_+(t), & \text{if } \alpha_+ \leq t \leq \beta_+,  \\ \label{15}
r(M)t, & \text{if } \beta_+ \leq t \leq 0. 
\end{cases}
\end{eqnarray}
\noindent Importantly,   \( z_+(t) \)  with $M=\mathbf  M$ and $m = \mathbf m$ provides an upper bound for the periodic solution \( x(t) \) on \([-1, 0]\).
\begin{lemma}\label{L2.5} Assume that $a \in \left (-2, 0 \right)$.
Then  \( x(t) \leq z_+(t) \) for all \( t \in [-1, 0] \).
\end{lemma}
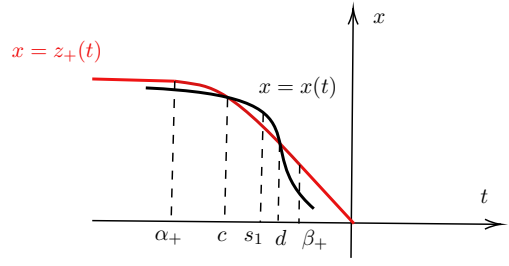
\begin{figure}[h] \label{F3}
\centering
\scalebox{0.75}
{

\tikzset{every picture/.style={line width=0.75pt}} 

\begin{tikzpicture}[x=0.75pt,y=0.75pt,yscale=-1,xscale=1]

\draw    (305.5,187.5) -- (578.5,189.98) ;
\draw [shift={(580.5,190)}, rotate = 180.52] [color={rgb, 255:red, 0; green, 0; blue, 0 }  ][line width=0.75]    (10.93,-3.29) .. controls (6.95,-1.4) and (3.31,-0.3) .. (0,0) .. controls (3.31,0.3) and (6.95,1.4) .. (10.93,3.29)   ;
\draw    (479.5,212.5) -- (480.99,47) ;
\draw [shift={(481,45)}, rotate = 90.35] [color={rgb, 255:red, 0; green, 0; blue, 0 }  ][line width=0.75]    (10.93,-3.29) .. controls (6.95,-1.4) and (3.31,-0.3) .. (0,0) .. controls (3.31,0.3) and (6.95,1.4) .. (10.93,3.29)   ;
\draw [color={rgb, 255:red, 232; green, 18; blue, 18 }  ,draw opacity=1 ][line width=1.5]    (445,150) -- (481.5,189.5) ;
\draw [color={rgb, 255:red, 239; green, 20; blue, 20 }  ,draw opacity=1 ][line width=1.5]    (305.5,92.5) -- (361,94) ;
\draw [color={rgb, 255:red, 215; green, 19; blue, 19 }  ,draw opacity=1 ][line width=1.5]    (361,94) .. controls (382.5,97.5) and (392.5,92.5) .. (445,150) ;
\draw  [dash pattern={on 4.5pt off 4.5pt}]  (361,94) -- (358.5,189.5) ;
\draw  [dash pattern={on 4.5pt off 4.5pt}]  (445,150) -- (444.5,189.5) ;
\draw    (315.5,187.5) ;
\draw [line width=1.5]    (341.5,98.5) .. controls (470.5,104.5) and (407.5,138.5) .. (454.5,179.5) ;
\draw  [dash pattern={on 4.5pt off 4.5pt}]  (396.5,105.5) -- (394.5,187.5) ;
\draw  [dash pattern={on 4.5pt off 4.5pt}]  (431.5,135.5) -- (430.5,187.5) ;
\draw  [dash pattern={on 4.5pt off 4.5pt}]  (420.5,115.5) -- (418.5,188.5) ;

\draw (250,65.4) node [anchor=north west][inner sep=0.75pt]  [color={rgb, 255:red, 239; green, 8; blue, 8 }  ,opacity=1 ]  {$x=z_{+}( t)$};
\draw (565,165.4) node [anchor=north west][inner sep=0.75pt]    {$t$};
\draw (493,47.4) node [anchor=north west][inner sep=0.75pt]    {$x$};
\draw (346,192.4) node [anchor=north west][inner sep=0.75pt]    {$\alpha _{+}$};
\draw (445,192.15) node [anchor=north west][inner sep=0.75pt]    {$\beta _{+}$};
\draw (416,89.4) node [anchor=north west][inner sep=0.75pt]    {$x=x(t)$};
\draw (388,193.4) node [anchor=north west][inner sep=0.75pt]    {$c$};
\draw (427,192.4) node [anchor=north west][inner sep=0.75pt]    {$d$};
\draw (406,193.4) node [anchor=north west][inner sep=0.75pt]    {$s_{1}$};

\end{tikzpicture}

}
\vspace{0mm}

\caption{\hspace{0cm} See the proof of Lemma \ref{L2.5}}
\end{figure} 
\begin{proof}
Clearly, due to \eqref{13}, this relationship only needs to be proven for \( t \in [-1, 0] \cap [\alpha_+, \beta_+] \). Suppose that \( x(s_1) > z_+(s_1) \) for some \( s_1 \in [\alpha_+, \beta_+] \).
In this case, there will exist $c,d,$ \( \alpha_+ \leq c \leq d \leq \beta_+ \) (see Fig. 5) such that  
\[
x(c) = z_+(c), \quad x(d) = z_+(d),
\]
\[
x'(c) \geq z_+'(c), \quad x'(d) \leq z_+'(d)<0.
\]
First, suppose that \( x(s) \) decreases between \( c \) and \( d\) (i.e.  $t_{max} \not\in (c,d)$), so we have from \eqref{14}  that
\begin{equation}\label{z}
x''(s) \geq a  r(\mathbf  m) = z_+''(s).
\end{equation}
Integrating the last equation between $c $ and $d$ we get 
$ x'(d) - x'(c) \geq z_+'(d) - z_+'(c)$ or $
$$ x'(d) - z_+'(d) \geq x'(c) - z_+'(c).$
But  from above \( x'(d) - z_+'(d) \leq 0 \) and \( x'(c) - z_+'(c) \geq 0 \), therefore
\[
 x'(d) = z_+'(d), \quad x'(c) = z_+'(c).
\]
Thus
\[
\int_c^d x''(s) \, ds = \int_c^d z_+''(s) ds.
\]
This together with \eqref{z} implies that 
$
 x''(s) = z_+''(s) \  \text{on } [c, d],$  and, consequently, $ x(s_1) = z_+(s_1), $ a contradiction.

Now, if  $t_{max} \in (c,d)$  then 
\[
0 = x'(t_{max}) > z_+'(t_{max}), \quad x'(d) \leq z_+'(d).
\]
Then, following the same argument with \( c \) replaced by \(t_{max} \), we again get a contradiction.
\end{proof}

Now, let \( z_{+i} \) be  defined by \eqref{312} for \( a = a_i \in (-2,0)\), $i=1,2,$ and the corresponding $\beta_{+j}$ are negative, 
\[
\beta_{+j} := \frac{1 + M}{a_j} + \frac{M  (m + 1)}{2  a_j  m(1 + M)} <0, \qquad i=1,2.
\]
Then the following monotonicity property holds: 
\begin{lemma}\label{2.6}
 Let $-2<a_2 <a_1<0$. Then \( z_{+1}(s) \leq z_{+2}(s) \), for \( s \in [-1, 0] \). 
\end{lemma}
\begin{figure}[h] \label{F3}
\centering
\scalebox{0.75}
{

\tikzset{every picture/.style={line width=0.75pt}} 

\begin{tikzpicture}[x=0.75pt,y=0.75pt,yscale=-1,xscale=1]

\draw    (305.5,187.5) -- (578.5,189.98) ;
\draw [shift={(580.5,190)}, rotate = 180.52] [color={rgb, 255:red, 0; green, 0; blue, 0 }  ][line width=0.75]    (10.93,-3.29) .. controls (6.95,-1.4) and (3.31,-0.3) .. (0,0) .. controls (3.31,0.3) and (6.95,1.4) .. (10.93,3.29)   ;
\draw    (479.5,222.5) -- (480.99,47) ;
\draw [shift={(481,45)}, rotate = 90.35] [color={rgb, 255:red, 0; green, 0; blue, 0 }  ][line width=0.75]    (10.93,-3.29) .. controls (6.95,-1.4) and (3.31,-0.3) .. (0,0) .. controls (3.31,0.3) and (6.95,1.4) .. (10.93,3.29)   ;
\draw [color={rgb, 255:red, 232; green, 18; blue, 18 }  ,draw opacity=1 ][line width=1.5]    (445,150) -- (481.5,189.5) ;
\draw [color={rgb, 255:red, 239; green, 20; blue, 20 }  ,draw opacity=1 ][line width=1.5]    (305.5,92.5) -- (361,94) ;
\draw [color={rgb, 255:red, 215; green, 19; blue, 19 }  ,draw opacity=1 ][line width=1.5]    (361,94) .. controls (382.5,97.5) and (392.5,92.5) .. (445,150) ;
\draw  [dash pattern={on 4.5pt off 4.5pt}]  (361,94) -- (358.5,189.5) ;
\draw  [dash pattern={on 4.5pt off 4.5pt}]  (445,150) -- (444.5,189.5) ;
\draw    (315.5,187.5) ;
\draw  [dash pattern={on 0.84pt off 2.51pt}]  (402,94.4) -- (480.4,96) ;
\draw    (468.4,166) -- (479.6,187) ;
\draw [line width=1.5]    (402,94.4) .. controls (418,94.8) and (450,127.6) .. (468.4,166) ;
\draw  [dash pattern={on 4.5pt off 4.5pt}]  (468.4,166) -- (469.2,188.8) ;
\draw  [dash pattern={on 4.5pt off 4.5pt}]  (402,94.4) -- (402,188) ;
\draw [line width=1.5]    (361,94) -- (402,94.4) ;

\draw (493,86.4) node [anchor=north west][inner sep=0.75pt]    {$M$};
\draw (289.2,71.4) node [anchor=north west][inner sep=0.75pt]  [font=\normalsize,color={rgb, 255:red, 239; green, 8; blue, 8 }  ,opacity=1 ]  {$x=z_{+,1}( t)$};
\draw (565,165.4) node [anchor=north west][inner sep=0.75pt]    {$t$};
\draw (493,47.4) node [anchor=north west][inner sep=0.75pt]    {$x$};
\draw (346.4,192.4) node [anchor=north west][inner sep=0.75pt]  [font=\normalsize]  {$\alpha _{+1}$};
\draw (420.6,193.75) node [anchor=north west][inner sep=0.75pt]  [font=\normalsize]  {$\beta _{+1}$};
\draw (384.4,192.8) node [anchor=north west][inner sep=0.75pt]  [font=\normalsize]  {$\alpha _{+2}$};
\draw (451,194.15) node [anchor=north west][inner sep=0.75pt]  [font=\normalsize]  { $\beta _{+2}$};
\draw (408.8,75) node [anchor=north west][inner sep=0.75pt]  [font=\normalsize,color={rgb, 255:red, 25; green, 24; blue, 24 }  ,opacity=1 ]  {$x=z_{+,2}( t)$};
\draw (354.4,183) node [anchor=north west][inner sep=0.75pt]    {$\bullet $};
\draw (397,183) node [anchor=north west][inner sep=0.75pt]    {$\bullet $};
\draw (440,184) node [anchor=north west][inner sep=0.75pt]    {$\bullet $};
\draw (464,184) node [anchor=north west][inner sep=0.75pt]    {$\bullet $};

\end{tikzpicture}

}
\vspace{0mm}

\caption{\hspace{0cm} Graphs of $x=z_{+,1}( t)$ and $x=z_{+,2}( t)$.}
\end{figure}
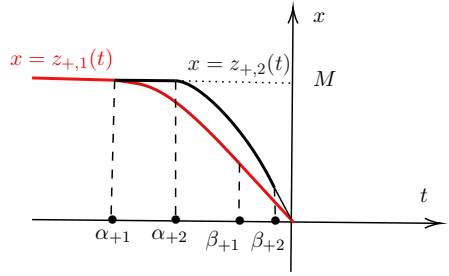 

\begin{proof}
Note that 
${\beta_{+1}}/{\beta_{+2}} = {a_2}/{a_1} > 1 $ and consequently $\beta_{+1} < \beta_{+2}.$
Similarly,
\[
\alpha_{+j}: = \frac{M}{r_{a_j}(M)} - \frac{r_{a_j}(M)}{2  a_j  r_{a_j}(m)} =  \frac{1 + M}{a_j} - \frac{M(1 + m)}{2 a_j m(1 + M)}.
\]
Thus, 
\begin{equation}\label{al}
{\alpha_{+1}}/{\alpha_{+2}} = {a_2}/{a_1} > 1,
\end{equation}
which implies $\alpha_{+1} < \alpha_{+2}$.
Clearly, from \eqref{312} we have that 
\[
z_{+2}(s) \geq z_{+1}(s), \quad \text{for } s \in [\alpha_{+1}, \alpha_{+2}] \cup [\beta_{+2},0].
\]

Next, assume  that  \( \alpha_{+2}< \beta_{+1} \).
We claim that the graphs of $z_{+1}(t)$ and $ z_{+2}(t)$ 
do not intersect on \(  [\alpha_{+2}, \beta_{+1}] \). Supposing the contrary,  
we  have for some \(s \in  [\alpha_{+2}, \beta_{+1}] \):
\[
M + 0.5{a_1 r_{a_1}(m)} (s - \alpha_{+1})^2 = M + 0.5{a_2 r_{a_2}(m)} (s - \alpha_{+2})^2,
\]
and, consequently, 
\[
  \frac{a_1^2}{a_2^2}= \frac{a_1 r_{a_1}(m)}{a_2 r_{a_2}(m)} = \frac{(s - \alpha_{+2})^2}{(s - \alpha_{+1})^2}.
\]
Thus $\frac{a_1}{a_2} = \frac{s - \alpha_{+2}}{s - \alpha_{+1}},$  which implies $a_1(s - \alpha_{+1}) = a_2(s - \alpha_{+2})$ and therefore  $a_1 \alpha_{+1} - a_2 \alpha_{+2} = (a_1  - a_2) s.$
However, it follows from \eqref{al} that \( a_1 \alpha_{+1} - a_2 \alpha_{+2} = 0 \) and  because of \( s \neq 0 \), we conclude that \( a_1 = a_2 \),  a contradiction.

Hence, regardless if   \( \alpha_{+2}< \beta_{+1} \) holds or not, we obtain that 
$z_{+2}(\beta_{+1}) > z_{+1}(\beta_{+1})$. 
Therefore, since $z_{+2}(t)$ is concave and $z_{+1}(t)$ is linear on  $[\beta_{+1},0]$, we conclude that 
\[
z_{+2}(s) > z_{+1}(s), \quad \text{for } s \in [\beta_{+1}, 0),
\]
which finalizes the proof.
\end{proof}

Equation \eqref{m} with the estimate from Lemma \ref{2.1},
$
 f(x) > r(x)$  for $  x > 0$,
implies
\[
\mathbf  m > \int_{-1}^0 r(x(s)) \, ds.
\]
By Remark \ref{R2.5}, the following function $L_a$ is well defined
for all $0<M\leq  A_a(m)$: 
\begin{equation} \label{La}
L_a(M, m):=\int_{-1}^0 r(z_+(s)) \, ds.    
\end{equation}
In this way, from Lemma \ref{L2.5} and the monotonicity of $r$, we get the following.
\begin{theorem}\label{T2.7} The critical values $\mathbf  m$ and $\mathbf  M$ of each slowly oscillating periodic solution  satisfy the inequality 
$
\mathbf  m >  L_a(\mathbf  M, \mathbf  m). 
$
\end{theorem}
Now, we define the region \( \mathcal{L}(a) \) depending on \( a = f'(0) \) by
\[
\mathcal{L}(a) = \{(M, m) \in \mathbb{R}_+ \times (-1, \sqrt{-2a/\pi} - 1) : m \geq L_a(M, m), \  M\leq  A_a(m)\}.
\]
As the next assertion shows,   \( \mathcal{L}(a) \) is monotone decreasing with respect to  $a<0$.
\begin{lemma}\label{2.8}
\( \mathcal{L}(a_1) \subset \mathcal{L}(a_2) \) for  \( a_2 < a_1 < 0 \).
\end{lemma}
\begin{proof} First, note that $A_{a_1}(m) < A_{a_2}(m)$ for $m>-1$. Next, since 
\( z_{+1}(s) \leq z_{+2}(s) \)  for \( s \in [-1, 0] \), we have that 
\[
r_{a_1}(z_{+1}(s)) = \frac{a_1  z_{+1}(s)}{1 + z_{+1}(s)} > \frac{a_2  z_{+1}(s)}{1 + z_{+1}(s)} 
\geq \frac{a_2  z_{+2}(s)}{1 + z_{+2}(s)} = r_{a_2}(z_{+2}(s)), \quad s \in [-1, 0].
\]
Thus,
\begin{equation}\label{L}
L_{a_1}(M, m) =
 \int_{-1}^0 r(z_{+1}(s)) \, ds
 > 
 \int_{-1}^0 r(z_{+2}(s)) \, ds,
=L_{a_2}(M, m)
\end{equation}
so that  $m \geq L_{a_2}(M, m)$ if $m > L_{a_1}(M, m)$.  Consequently,
$
\mathcal{L}(a_1) \subset \mathcal{L}(a_2).
$
\end{proof}
\begin{corollary}\label{2.9}
\( \mathcal{L}(a) \subset \mathcal{L}\left(-{37}/{24}\right) \) for all \( -{37}/{24} \leq a < 0 \).
\end{corollary}
\begin{lemma}\label{L2.10}
If  \( 0 < M_1 < M_2 \) and $\beta_+(M_j) < 0$, then \( L_a(M_1, m) > L_a(M_2, m) \). It also holds that $ L_a(0, m) = 0 $ and \( L_a(+\infty, m) = a + \ln\left(-a+1\right) 
\).
\end{lemma}
\begin{proof} For a fixed $m\in (-1, 0)$, let 
 \( z_{+1}(s)\) and \( z_{+2}(s) \) denote the functions defined by \eqref{312} for some fixed positive $M_1 <M_2$, respectively. 
In view of the monotonicity of \( r(x) \), it is enough to prove that \( z_{+1}(s) \leq z_{+2}(s) \) for \( s \in [-1, 0] \). 
Since 
\[
\alpha'_+(M) = \frac{1}{a} - \frac{1 + m}{2am}\frac{1}{(1+M)^2} < 0,
\]
we have that \( \alpha_+(M_1) > \alpha_+(M_2) \).  Now, as an immediate consequence of  the definitions of 
$z_{+j}(s)$ we get that   \(\sigma(t):= z_{+2}(s)- z_{+1}(s) >0 \) for all $s \leq \alpha_+(M_2)$ and 
$s \in [\beta_+(M_2),0)$ (note that $\sigma(0)=0$ and $\sigma'(0)= r(M_2) - r(M_1) <0$).  Furthermore, since the second derivatives of both parabolic fragments of $z_{+j}(s)$ are equal,  we have that $\sigma''(0) \leq 0$ on $[\alpha_+(M_2), \beta_+(M_2)]$ a.e. Thus $\sigma(s)$ is concave on $[\alpha_+(M_2), \beta_+(M_2)]$ which implies that
\( z_{+2}(s) > z_{+1}(s) \) for all \( s \in [-1, 0) \).

Furthermore, \( L_a(0, m) = 0 \), because if \( M = 0 \), then \( z_+(s) \equiv 0 \). Also, for \( M \to +\infty \), we have \( \alpha_+ \), \( \beta_+ \to -\infty \), \( r(M) \to a 
\), and thus
\[
L_a(+\infty, m) = \int_{-1}^0 r\left(a s\right) \, ds = -\int_0^{-a} \frac{u}{1 + u} \, du = a + \ln\left(-a+1\right). 
\]
Note that this value does not depend on \( m \).
\end{proof}
\begin{corollary}\label{C2.11}
If  \(a\in (0,-\frac{37}{24}]\) then  \(\mathbf  m > -\frac{37}{24} + \ln\left(\frac{37}{24}+1\right)=-0.6088\dots \)
\end{corollary}
\begin{proof} 
Suppose  that $\mathbf  m\leq  -0.6088\dots$ Then $\beta_+ <0$ in view of Remark \ref{R2.5}.  $L_a( M, \mathbf m)$ is decreasing in $M$ by 
the previous lemma.  This together with  Theorem \ref{T2.7} and \eqref{L} implies that
$\mathbf  m > L_{-37/24}(+\infty,\mathbf  m) = -0.6088\dots$, a contradiction.
\end{proof}
\begin{lemma}
If \( m_2 > m_1\) and $(M,m_j)\in \mathcal{L}(a)$, then \( L_a(M, m_2) \geq L_a(M, m_1) \). Moreover,  \( L_a(M, m_2) > L_a(M, m_1) \) if $\beta_{+2} > -1$. 
\end{lemma}
\begin{proof} Note that $\beta_{+1} < \beta_{+2}$. Since $z_+(t,m):=z_+(t)$ is non-increasing in $m$, $0\leq z_+(t,m_2)\leq z_+(t,m_1)$ for $t \in [-1,0]$,  $z_+(t,m_2)\not= z_+(t,m_1)$ when $\beta_{+2}>-1$,  and $r(z)$ is decreasing, the result follows from the definition of  $L_a(M, m)$, see \eqref{La}. 
\end{proof}
\begin{remark}\label{2.15}
We will use the explicit form of \( L_a(M, m) \). If $\alpha_+ \geq -1$, then 
\begin{equation}\label{28}
L_a(M, m) = \int^0_{\beta_+} r(r(M) s) \, ds + \int_{\alpha_+}^{\beta_+} r(z_+(s)) \, ds + \int^{\alpha_+}_{-1} r(M) \, ds = I_1+I_2+I_3, 
\end{equation}
where 
\[
I_3:=\int^{\alpha_+}_{-1} r(M) \, ds = r(M)(\alpha_+ +1) = 
 M - \frac{(r(M))^2}{2 a r(m)} + r(M),
\]
\begin{equation*}\label{35}
I_2:=\int_{\alpha_+}^{\beta_+} r(p_+(s)) \, ds = \int_{\alpha_+}^{\beta_+} \frac{a p_+(s)}{1 + p_+(s)} \, ds = a (\beta_+ - \alpha_+) - a \int_{\alpha_+}^{\beta_+} \frac{ds}{1 + p_+(s)}.
\end{equation*}
Set $A^2 := -2(1+M)/(a r(m)) > 0$.  We have  
\[
\int_{\alpha_+}^{\beta_+} \frac{ds}{1 + p_+(s)}  = 
\frac{2}{a r(m)} \int_{\alpha_+}^{\beta_+} \frac{ds}{(s - \alpha_+)^2 + \frac{2(1+M)}{a r(m)}}
= \frac{2}{a r(m)} \int_{\alpha_+}^{\beta_+} \frac{ds}{(s - \alpha_+)^2 - A^2}
\]
\[
= \frac{1}{a r(m) A} \left[\ln\left|\frac{u - A}{u + A}\right|\right]_0^{\beta_+ - \alpha_+} 
= \frac{1}{a r(m) A} \ln\left|\frac{r(M) - a r(m) A}{r(M) + a r(m) A}\right|,
\]
where \( a r(m) A = -\sqrt{-2 a r(m)(1 + M)} \).
As a result, 
\[
I_2= \frac{r(M)}{r(m)} - \frac{a}{\sqrt{-2a r(m)(1 + M)}}  \ln\left|\frac{r(M) - \sqrt{-2a r(m)(1 + M)}}{r(M) + \sqrt{-2a r(m)(1 + M)}}\right|.
\]
Finally,
\[I_1:=
\int^0_{\beta_+} r(r(M) s) \, ds = -\frac{a M}{r(M)} - \frac{r(M)}{2 r(m)} + \frac{a}{r(M)} \ln\left( 1 + M + \frac{(r(M))^2}{2a r(m)} \right). 
\]
Similarly, explicit representations of \( L_a(M, m) \) in terms of elementary functions can be obtained for other values of \( M, m \), see Appendices A, B.  
\end{remark}

Combining Theorem \ref{T2.7}, inequalities (\ref{Am}), \eqref{L}   and Lemma \ref{L2.10}, 
we find that each admissible pair $(\mathbf  M, \mathbf m)$ satisfies the inequalities
$$\mathbf  m > L_a(\mathbf  M, \mathbf  m) > L_{-37/24}(A_{-37/24}(\mathbf  m),\mathbf  m), \quad  \text{for } a \in [-37/24,0).$$
In proving the main result of this section, Theorem \ref{C12}, the following rigorously validated  inequalities are instrumental. 
\begin{lemma} \label{L2.15}

(a) $m< L_{-37/24}(-m,m)$ for all $m \in  [-0.61,-0.009]$;

(b)  $\partial  L_{-37/24}(M,m)/\partial m \in [0,0.91)$ for all $(M,m)$ in the set $$\Pi:= \{(M,m): -m\leq M \leq A_{-37/24}(m);\  m \in  [-0.25,-0.009]\};$$

(c) $m > L_{-37/24}(A_{-37/24}(m),m)=:C(m),$ \ for  $m \in [-1/9,-0.009].$

(d) $m < L_{-37/24}(A_{-37/24}(m),m)$  on $[-0.61,-0.25]$, \\ so that  \(\mathbf   m > -0.25 \) and \(\mathbf  M < 0.377\) for $a \in [-37/24,0)$.
\end{lemma}
\begin{proof} Below, we use the inequality $-m < A_{a}(m)$ satisfied by all $m \in [-0.61,0)$, for $a \in [-37/24,-3/2]$. The proof of this fact is elementary and therefore omitted.

(a) With $p(t)=t$,  $q(t)=L(-t,t)$ and $[t_*,t^*]=[-0.61,-0.009]$, all conditions of Lemma \ref{BL} are satisfied and  the inequality in a)  is established by running  the second script in Appendix A, with $J_0= 99$. 

(b) This estimate can be obtained by using verified bounds of  $\partial  L_{-37/24}(M,m)/\partial m$ on suitable refinements of the domain $\Pi$. 
See the MATLAB/INTLAB scripts in Appendix B.

(c) Note here that $A_a$ is a strictly decreasing function \cite{LPRTT}, which implies that the composed function  $q(t):=-C(-t)$ is strictly increasing on $[0.009,1/9]$. Thus we can apply  Lemma \ref{BL} on  $[t_*,t^*]=[0.009,1/9]$, with   $p(t)=t$.
See the third script in Appendix A,  it gives $J_0=67$.  

(d) It suffices to take $q(t):=C(t)$, $p(t)=t$, $[t_*,t^*]=[-0.61,-0.25]$ and to apply  Lemma \ref{BL}. See the last script  in Appendix A,  it gives $J_0=38$.  

Finally, each admissible $\mathbf  m<0$ should satisfy the inequality 
$ \mathbf  m >
C(\mathbf  m).  
$
 Thus  $\mathbf m > -0.25$ so that $\mathbf M < A_{-37/24}(\mathbf m) < A_{-37/24}(-0.25) < 0.377$. 
\end{proof}

\begin{theorem}\label{C12}
Suppose \( a = -{37}/{24} \). For each \( m \in [-0.25, 0) \), the equation \( m = L(M, m) \) (for simplicity, we omit subindex $a=-{37}/{24}$) has a unique positive solution \( M = \hat{L}(m) \). \( \hat{L}(m) \) is  a strictly decreasing continuous function and 
  $\hat{L}(m) >-m$ for all $m \in [-0.25,-0.009].$ 
Finally, the sequence  of strictly decreasing continuous functions $m_n(M)$, $M \in [0.0094, 0.377]$,  recursively defined by
\begin{equation}\label{MMM}
m_{n+1}(M) = L(M,m_n(M)), \ n =0,1,2,\dots, \quad m_0(M) =-M,    
\end{equation}
is non-decreasing and  converges uniformly to $\hat L^{-1}(M)$: $m_n(M) < \hat L^{-1}(M)$ for all $n$ and $M\in [0.0094, 0.377]$. 
\end{theorem}

It can be proved that \( \hat{L}(0^-) = 0 \) and \( \hat{L}'(0^-) = -1 \), see Appendix 1 and Fig. 7. Clearly, Theorem \ref{C12} shows that 
\[
\mathcal{L}\left(-{37}/{24}\right) = \{(M, m):  \hat{L}(m) \leq M \leq A(m),\ m \in [-0.25, -0.0093\dots)\}.
\]

 \vspace{0mm}
 
\begin{figure}[h]\label{Fig5}
\centering
\scalebox{0.6}
{
 {\includegraphics[width=11cm]{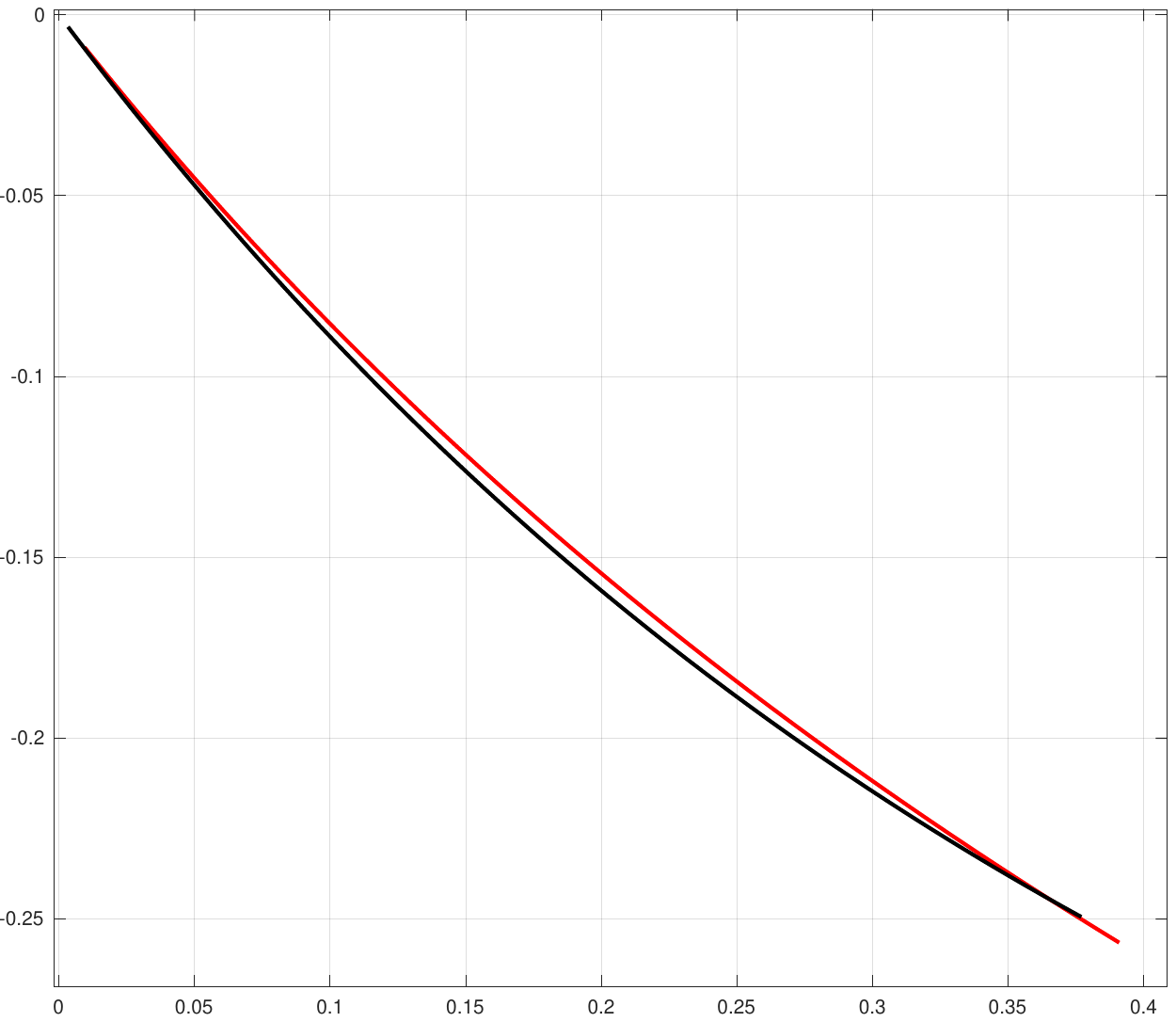}}
 }

 \vspace{-10mm}
\caption{Domain \( \mathcal{L}\left(-{37}/{24}\right)\) bounded by the graphs of $\hat L$ (in black) and $A$ (in red). } 
\end{figure} 

\begin{remark} \label{R2.18}The interval technique of the INTLAB toolbox allows us to compute rigorous lower bound for $m_n(M)$. Actually, for our purposes it suffices to estimate $m_5(M)$, see Appendix C
for the corresponding MATLAB script. Note also that the verified  estimate  $m_5(0.009401)>-0.0093$ is used in our work. It implies that $\mathbf M \geq \hat L(-0.0093)>0.009401$ for all admissible values of $\mathbf M$.  Also, $L(0.0094,-0.009) =
  -0.0092701897... <-0.009$ so that $-0.009 > \hat L^{-1}( 0.0094)$.
\end{remark}
\begin{proof}
Uniqueness and existence of \( M = \hat{L}(m) \) is the consequence of the monotonicity of \( L(\cdot, m) \) and the relations \( L(0, m) = 0 \) and 
\( L(A(m), m) <m \) (if $m \in [-1/9,-0.009]$, ) or 
\( L(+\infty, m) =-0.6088 <m \) (if $m \in [-0.25, -1/9]$, cf. Remark \ref{R2.5}). 
The continuity of \( \hat{L} \) follows easily from the continuity of \( L(M, m) \) and the uniqueness of the solution of \( m = L(M, m) \) for each fixed \( m \).  

Next, $\hat{L}(m) >-m$, $m \in  [-0.61,-0.009]$,   in view of Lemma \ref{L2.15} a). The assertion b) of the same lemma shows that $\hat{L}$ is an injection and therefore strictly monotone. Indeed, if $\hat{L}(m_1)=\hat{L}(m_2)$ then $m_1=m_2$ because of
$$|m_1-m_2|=|L(\hat{L}(m_1),m_1)-L(\hat{L}(m_1),m_2)| \leq 0.91|m_1-m_2|. $$

Finally, since the function $m_0(M)=-M$ is continuous and strictly decreasing,  from the continuity and mixed monotonicity of $L(M,m)$, we find that each $m_j(M)$ has the same properties for all $j \in \N$. Next, $m_0(M)=-M < L(M,-M)=m_1(M) \leq \hat L^{-1}(M)$ for all $M \in [0.0094, 0.377]$. Therefore $m_2(M)=L(M,m_1(M)) \geq L(M,m_0(M)) =m_1(M)$ on the same interval. Repeating this argument, we find that the sequence of negative numbers $m_j(M)$ is non-decreasing and $\hat L(m_j(M)) \geq M$. Setting $m^*(M)= \lim m_j(M)$, we obtain $\hat L(m^*(M)) \geq M$ and $m^*(M)=L(M,m^*(M))$. 
Therefore  $\hat L (m^*(M))=M$ with $m^*(M) <0$ for all $M >0$. The uniform convergence now follows from the Dini's monotone convergence theorem. 
\end{proof}

\section{A bound for the maximum of a slowly oscillating periodic solution} \label{Sec3}
In this section, considering the slowly oscillating periodic solution $x(t)$, we will obtain an upper bound for \( \mathbf M = \max_{t \in \mathbb{R}} x(t) \) in the form of  an implicit relation between $\mathbf M$ and \( \mathbf  m = \min_{t \in \mathbb{R}} x(t) \).
 This case is more complex than the previous one considered in Section 3 because it requires evaluating \( \max_{x \in [m, 0]} f'(x) \) instead of  relying on the simple equality \( \min_{x \in [0, M]} f'(x) = f'(0) \).

\vspace{0mm}

 \begin{figure}[h] \label{F3}
\centering
\scalebox{0.75}
{

\tikzset{every picture/.style={line width=0.75pt}} 

\begin{tikzpicture}[x=0.75pt,y=0.75pt,yscale=-1,xscale=1]

\draw    (64.5,153.5) -- (636.5,156.49) ;
\draw [shift={(638.5,156.5)}, rotate = 180.3] [color={rgb, 255:red, 0; green, 0; blue, 0 }  ][line width=0.75]    (10.93,-3.29) .. controls (6.95,-1.4) and (3.31,-0.3) .. (0,0) .. controls (3.31,0.3) and (6.95,1.4) .. (10.93,3.29)   ;
\draw [line width=2.25]    (62.5,70.5) .. controls (223.5,-29.5) and (278.5,460.5) .. (534.5,93.5) ;
\draw  [dash pattern={on 4.5pt off 4.5pt}]  (597.5,57) -- (596.5,156) ;
\draw    (355.5,228.5) -- (356.5,152.5) ;
\draw    (236.5,254.5) -- (236.5,12.5) ;
\draw [shift={(236.5,10.5)}, rotate = 90] [color={rgb, 255:red, 0; green, 0; blue, 0 }  ][line width=0.75]    (10.93,-3.29) .. controls (6.95,-1.4) and (3.31,-0.3) .. (0,0) .. controls (3.31,0.3) and (6.95,1.4) .. (10.93,3.29)   ;
\draw  [dash pattern={on 4.5pt off 4.5pt}]  (153.5,55) -- (597.5,57) ;
\draw [color={rgb, 255:red, 208; green, 2; blue, 27 }  ,draw opacity=1 ][line width=2.25]    (174.5,78.5) -- (288,216) ;
\draw  [dash pattern={on 4.5pt off 4.5pt}]  (234,232) -- (410,232) ;
\draw [color={rgb, 255:red, 237; green, 18; blue, 18 }  ,draw opacity=1 ][line width=2.25]    (59.5,57.5) -- (134.5,55.5) ;
\draw [line width=2.25]    (534.5,93.5) .. controls (588.5,31.5) and (634.5,67.5) .. (658.5,89.5) ;
\draw [color={rgb, 255:red, 234; green, 13; blue, 13 }  ,draw opacity=1 ][line width=2.25]    (448.5,210.5) -- (487.5,154.5) ;
\draw [color={rgb, 255:red, 246; green, 37; blue, 37 }  ,draw opacity=1 ][fill={rgb, 255:red, 239; green, 11; blue, 11 }  ,fill opacity=1 ][line width=2.25]    (312.5,232.5) -- (410,232) ;
\draw [color={rgb, 255:red, 229; green, 7; blue, 7 }  ,draw opacity=1 ][line width=3.75]  [dash pattern={on 3.75pt off 3pt on 7.5pt off 1.5pt}]  (134.5,55.5) .. controls (146.5,56.5) and (150.5,56.5) .. (174.5,78.5) ;
\draw [color={rgb, 255:red, 218; green, 30; blue, 30 }  ,draw opacity=1 ][line width=3]  [dash pattern={on 3.75pt off 3pt on 7.5pt off 1.5pt}]  (288,216) .. controls (298.5,225.5) and (297.5,230.5) .. (312.5,232.5) ;
\draw [color={rgb, 255:red, 237; green, 18; blue, 18 }  ,draw opacity=1 ][line width=3]  [dash pattern={on 3.75pt off 3pt on 7.5pt off 1.5pt}]  (410,232) .. controls (425.5,232.5) and (443.5,216.5) .. (451.5,205.5) ;

\draw (243,132.4) node [anchor=north west][inner sep=0.75pt]    {$0$};
\draw (352,130.4) node [anchor=north west][inner sep=0.75pt]    {$1$};
\draw (487,58.4) node [anchor=north west][inner sep=0.75pt]    {$x=x( t)$};
\draw (631,127.4) node [anchor=north west][inner sep=0.75pt]    {$t$};
\draw (590,161.4) node [anchor=north west][inner sep=0.75pt]    {$t_{m}$};
\draw (241,33.4) node [anchor=north west][inner sep=0.75pt]    {$\mathbf M$};
\draw (152,30.4) node [anchor=north west][inner sep=0.75pt]  [color={rgb, 255:red, 218; green, 13; blue, 13 }  ,opacity=1 ]  {$x=z_{+}( t)$};
\draw (213,225) node [anchor=north west][inner sep=0.75pt]    {$\bf m$};
\draw (488.5,159.9) node [anchor=north west][inner sep=0.75pt]    {$t_{m} -1$};
\draw (423,236.4) node [anchor=north west][inner sep=0.75pt]    {$\textcolor[rgb]{0.92,0.05,0.05}{x=\tilde{z}( t)}$};
\draw (245,239.4) node [anchor=north west][inner sep=0.75pt]  [color={rgb, 255:red, 218; green, 13; blue, 13 }  ,opacity=1 ]  {$x=z_{-}( t)$};

\end{tikzpicture}

}
\vspace{-40mm}

\caption{\hspace{0cm} Graphs of solution $x(t)$ (in black)  and its bounding functions $x=z_{+}(t)$, $x=z_{-}(t)$, $x=\tilde z(t)$ (in red). Parabolic parts of the bounding graphs are shown as dashed curves, their linear parts are represented by continuous lines.}
\end{figure}
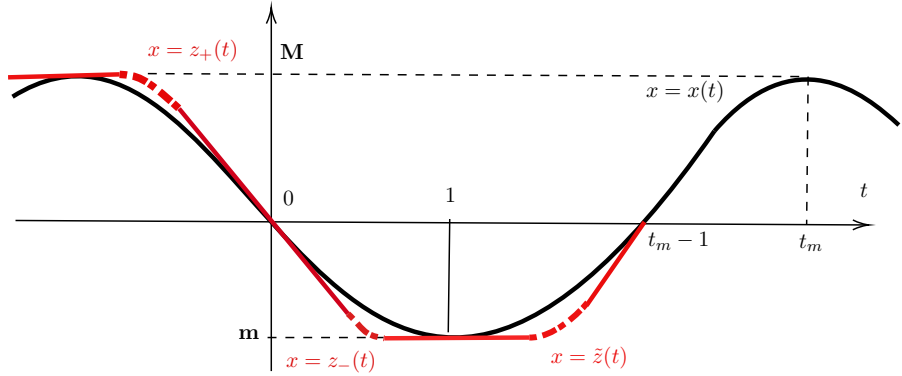

Our strategy for establishing the estimate for 
$\mathbf  M=x(t_m)$ is as follows (see Fig. 8). 
Since $x(t_m-1)=0$, integrating equation \eqref{4} from $t_m-1$ to $t_m$, yields 
\begin{equation*}
\mathbf  M= \int_{t_m-2}^{t_m-1} f(x(s))ds.    
\end{equation*}
Thus we  need an appropriate lower bound $\tilde z(t)$
for $x(t)$ for $t$ between $t_m-2$ to $t_m-1$. Again, we will obtain this  bound  as  piece-wise continuous combination of linear functions (continuous red lines in Fig. 8) and  quadratic functions (shown as dashed red curves). However, unlike the construction in Section \ref{Sec2}, where the leading  coefficient of the approximating parabola depended  only on \( m \),  
 it will  depend now   on both \(m\) and \(M\). Moreover, to  determine this coefficient, we have to use a suitable lower bound $z_-(t)$ for the descending arc of the solution $x(t)$ on the interval $[0,1]$. The construction of $z_-(t)$ is very similar to the construction of $z_+(t)$ in Section \ref{Sec2} except for a few details. For  completeness, it is given below.   
We have
\begin{equation}\label{38}
x'(s) = f(x(s - 1)) \geq r(x(s - 1)) \geq r(\mathbf M), \quad \text{if} \,\, s \in [0, 1], 
\end{equation}
so that \( x(s) \geq r(\mathbf  M) s \), and therefore \( x(s) \geq \max\{r(\mathbf  M) s, \mathbf  m\} \) for \( s \geq 0 \).

Now, for \( s \in [0, 1] \), we have:
\[
x''(s) = f'(x(s - 1))\;  x'(s - 1) = f'(x(s - 1)) \; f(x(s - 2)),
\]
where \( x(s - 1) \geq 0 \). Aiming to find an upper bound for \( x''(s) \), we are interested in the case where \( f(x(s - 2)) \leq 0 \), so using \eqref{dd} it is easy to see that:
\[
x''(s) \leq f'(0) r(\mathbf  M) = a r(\mathbf  M),\quad  s \in [0, 1]. 
\]

The function \( z_-(s) \), $s \in [0,1]$, is defined similarly to the  function \( z_+(s) \), \(s \in [-1, 0] \):
\[
z_-(s) =
\begin{cases}
r(M) s, & \text{if } 0 \leq s \leq \beta_-, \\
m + \frac{a r(M)}{2}  (s - \alpha_-)^2, & \text{if } \beta_- \leq s \leq \alpha_-, \\
m, & \text{if } \alpha_- \leq s,
\end{cases}
\]
where
\begin{equation}\label{39}
\beta_- = \frac{1}{2a} + \frac{m}{r(M)}, \quad  \alpha_- = -\frac{1}{2a} + \frac{m}{r(M)}.
\end{equation}
As we will see in Lemma \ref{3.2}, $z_-(s)$ with $M= \bf M$ and $m= \bf m$ serves as a lower bound for $x(s)$ on $[0,1]$. 
\begin{lemma}\label{L3.1}
$\beta_{-} >0$  if $a \in \left (-2, 0 \right)$ and $M\leq  A_a(m)$. 
\end{lemma}
\begin{proof} 
Indeed, \( \beta_- \leq 0 \) is equivalent to $\frac{m}{r(M)} \leq -\frac{1}{2a}$ and therefore
\[
\quad m \geq -\frac{1}{2} \frac{M}{M + 1} > -\frac{1}{2}.
\] Thus,  \( M(2m + 1) \geq -2m \) and \( M \geq -\frac{2m}{2m + 1} \).
In view of  \eqref{23}, we have that
\[
-\frac{2m}{2m + 1} \leq M < -1 -  \frac{2m}{1 + m} + \frac{1 + m}{m}  \ln(1 + m).
\]

However, an easy analysis  shows that
\[
\tau(m) = \ln(1 + m) - \frac{m}{1 + m}  \left(1 +   \frac{2m}{1 + m} - \frac{2m}{2m + 1}\right) > 0,
\]
for \( m \in \left( -{1}/{2}, 0 \right) \).
This leads to a contradiction, which proves that \( \beta_- > 0 \).
\end{proof} 
\begin{lemma}\label{3.2}
The function $z_-(s)$  is a lower bound for the solution \( x(s)\),i.e.
\( x(s) \geq z_-(s) \), for \( s \in [0, 1] \). {Thus, \( \alpha_- \leq 1 \) for all admissible $(\mathbf  M,\mathbf  m)$.}

\end{lemma}
\begin{proof}

By virtue of \eqref{38}, it suffices to prove this inequality on \( [\beta_-, \alpha_-] \) assuming that $\beta_- <1$. Suppose there exists \( s_1 \in [\beta_-, \alpha_-] \) such that \( x(s_1) < z_-(s_1) \). 

Assume first that  \( \alpha_- \leq 1 \). Then there exist \( a \) and \( b \) such that \( \beta_- \leq b < a \leq \alpha_- \), where \( x(b) = z_-(b) \), \( x(a) = z_-(a) \), \( x'(b) \leq z'_-(b) \), and \( x'(a) \geq z'_-(a) \).
Additionally, due to the construction of \( z_-(s) \), we have \( x''(s) \leq z''_-(s) \), for \( s \in [b, a] \).
Thus,
\[
\int_b^a x''(s) \, ds \leq \int_b^a z''_-(s) \, ds \quad 
\Longrightarrow \quad x'(a) - z'_-(a) \leq x'(b) - z'_-(b).
\]
But \( x'(a) - z'_-(a) \geq 0 \) and \( x'(b) - z'_-(b) \leq 0 \), so we have:
\[
\int_b^a x''(s) \, ds = \int_b^a z''_-(s) \, ds \ \Longrightarrow \ x''(s) = z''_-(s)
\Longrightarrow \ x(s) = z_-(s), \ s \in [a, b].
\]
Thus, \( x(s_1) = z_-(s_1) \).
Therefore, we have a contradiction.

Now, if \( \alpha_- > 1 \), then \( x'(1) = 0 > z'_-(1)\) and \( z_-(1) > m = x(1) \). Therefore 
there exists some $b \in (\beta_-,1)$ such that  \( x'(b) \leq z'_-(b) \).  
Now, since \( x''(s) \leq z''_-(s) \) for \( s \in [b, 1] \), we have:
\[
\int_b^1 x''(s) \, ds \leq \int_b^1 z''_-(s) \, ds \quad \Longrightarrow \ 0 < x'(1) - z'_-(1) \leq x'(b) - z'_-(b) \leq 0,
\]
a contradiction proving that actually 
 \( \alpha_- \leq 1 \) for all admissible pairs $(\mathbf  M, \mathbf m)$. 
\end{proof}

Note that, {in general, if we consider $z_-(s)$ by itself (as we do in the next lemma),  it can happen that  \( \alpha_- > 1 \).  In this case, however, the values of $(M,m)$ do not correspond to any slowly oscillating periodic solution. Actually, Lemmas \ref{L2.4}, \ref{L3.1} and Remark \ref{R2.5} show that for $a\in (-2,0)$ all four functions $\alpha_\pm$, $\beta_\pm$ are well defined and $\beta_+<0$, $\beta_- >0$ in the domain $
 \{(M, m): 0< M \leq A_a(m),\ m  \in (-1,0)\}.
$
 }
\begin{lemma}\label{3.3}
\( z_{-2}(s) \leq z_{-1}(s) \), for all \( s \in [0, 1] \),  if \( a_2 < a_1 < 0 \), where
\[
z_{-j}(s) = 
\begin{cases}
r_{a_j}(M)s, & \text{if } s \in [0, \beta_{-j}], \\
m + \frac{a_j r_{a_j}(M)}{2}  (s - \alpha_{-j})^2, & \text{if } s \in [\beta_{-j}, \alpha_{-j}], \\
m, & \text{if } s \in [\alpha_{-j}, +\infty),
\end{cases}
\]
with
\[
\beta_{-j} = \frac{1}{2a_j} + \frac{m}{r_{a_j}(M)}, \quad \alpha_{-j} = -\frac{1}{2a_j} + \frac{m}{r_{a_j}(M)},\; j=1,2.
\]
\end{lemma}
\begin{proof}
\vspace{0mm}

 \begin{figure}[h] \label{F3}
\centering
\scalebox{0.75}
{

\tikzset{every picture/.style={line width=0.75pt}} 

\begin{tikzpicture}[x=0.75pt,y=0.75pt,yscale=-1,xscale=1]

\draw    (200.33,154.17) -- (530.33,155.16) ;
\draw [shift={(532.33,155.17)}, rotate = 180.17] [color={rgb, 255:red, 0; green, 0; blue, 0 }  ][line width=0.75]    (10.93,-3.29) .. controls (6.95,-1.4) and (3.31,-0.3) .. (0,0) .. controls (3.31,0.3) and (6.95,1.4) .. (10.93,3.29)   ;
\draw    (237.5,296.5) -- (237.5,94.5) ;
\draw [shift={(237.5,92.5)}, rotate = 90] [color={rgb, 255:red, 0; green, 0; blue, 0 }  ][line width=0.75]    (10.93,-3.29) .. controls (6.95,-1.4) and (3.31,-0.3) .. (0,0) .. controls (3.31,0.3) and (6.95,1.4) .. (10.93,3.29)   ;
\draw  [dash pattern={on 4.5pt off 4.5pt}]  (240.5,287.5) -- (525.5,288) ;
\draw [color={rgb, 255:red, 74; green, 144; blue, 226 }  ,draw opacity=1 ][line width=1.5]    (238,153) -- (263.5,235.5) ;
\draw [color={rgb, 255:red, 74; green, 144; blue, 226 }  ,draw opacity=1 ][line width=1.5]    (263.5,235.5) .. controls (287.5,278.5) and (317.5,287.5) .. (365.5,287.5) ;
\draw [line width=1.5]    (365.5,287.5) -- (525.5,288) ;
\draw [color={rgb, 255:red, 225; green, 16; blue, 16 }  ,draw opacity=1 ][line width=1.5]    (238,153) -- (287,235) ;
\draw [color={rgb, 255:red, 230; green, 28; blue, 28 }  ,draw opacity=1 ][line width=1.5]    (287,235) .. controls (309.5,276.5) and (355.5,285.5) .. (400.5,287.5) ;
\draw  [dash pattern={on 4.5pt off 4.5pt}]  (263.5,235.5) -- (262.5,155.5) ;
\draw  [dash pattern={on 4.5pt off 4.5pt}]  (288.5,158) -- (287,235) ;
\draw  [dash pattern={on 4.5pt off 4.5pt}]  (402.5,155.5) -- (400.5,287.5) ;
\draw  [dash pattern={on 4.5pt off 4.5pt}]  (366.5,155.5) -- (365.5,287.5) ;

\draw (217,129.73) node [anchor=north west][inner sep=0.75pt]    {$0$};
\draw (519,133.73) node [anchor=north west][inner sep=0.75pt]    {$t$};
\draw (198,277.73) node [anchor=north west][inner sep=0.75pt]    {$m$};
\draw (238,127.73) node [anchor=north west][inner sep=0.75pt]    {$\beta _{-2}$};
\draw (281,128.73) node [anchor=north west][inner sep=0.75pt]    {$\beta _{-1}$};
\draw (257,150.73) node [anchor=north west][inner sep=0.75pt]    {$\bullet $};
\draw (284,150.73) node [anchor=north west][inner sep=0.75pt]    {$\bullet $};
\draw (341,129.83) node [anchor=north west][inner sep=0.75pt]    {$\alpha _{-2}$};
\draw (390.03,129.68) node [anchor=north west][inner sep=0.75pt]  [rotate=-0.22]  {$\alpha _{-1}$};
\draw (362,150.73) node [anchor=north west][inner sep=0.75pt]    {$\bullet $};
\draw (397,150.73) node [anchor=north west][inner sep=0.75pt]    {$\bullet $};
\draw (240,265.73) node [anchor=north west][inner sep=0.75pt]  [color={rgb, 255:red, 5; green, 52; blue, 106 }  ,opacity=1 ]  {$z_{-2}( t)$};
\draw (302,232.73) node [anchor=north west][inner sep=0.75pt]  [color={rgb, 255:red, 208; green, 2; blue, 27 }  ,opacity=1 ]  {$z_{-1}( t)$};

\end{tikzpicture}

}
\vspace{-0mm}

\caption{\hspace{0cm} Graphs of the functions $z_{-2}(t)$ (in blue)  and $z_{-1}(t)$ (in red). }
\end{figure}
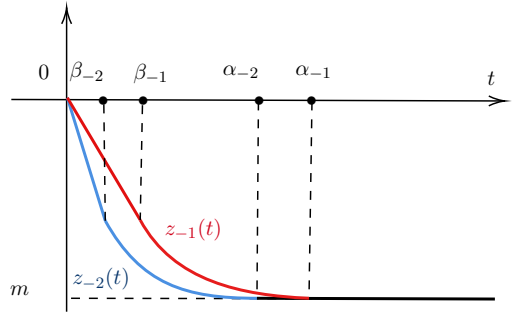 

It follows from the above definitions that 
\begin{equation}\label{a1}
{\beta_{-2}}/{\beta_{-1}} =  
{\alpha_{-2}}/{\alpha_{-1}} = {a_1}/{a_2} < 1.   
\end{equation}
Therefore
$
\beta_{-2} < \beta_{-1}$ and $ \alpha_{-2} < \alpha_{-1}.
$
Consequently, there is no intersection of the graphs of \( z_{-2}(s) \) and \( z_{-1}(s) \) on the intervals \( (0, \beta_{-2}] \) and \( [\alpha_{-2}, \alpha_{-1}) \).

There is also no intersection between \( z_{-1} \) and \( z_{-2} \) on \( [\beta_{-2}, \beta_{-1}] \), due to the convexity of the parabolic part of \( z_{-2} \) and the inequalities: \( z_{-2}(\beta_{-i}) < z_{-1}(\beta_{-i})\), $i=1,2$.

Suppose now that these graphs intersect at some point  \(s>0, \) with \( s\in  [\beta_{-1}, \alpha_{-2}) \), i.e.,  \( z_{-1}(s) = z_{-2}(s) \). Then  this point is  the intersection of the corresponding parabolas, i.e. 
\[
m + 0.5{a_1 r_{a_1}(M)}  (s - \alpha_{-1})^2 = m + 0.5{a_2 r_{a_2}(M)}(s - \alpha_{-2})^2,
\]
which implies 
$
 a_1^2 (s - \alpha_{-1})^2 = a_2^2 (s - \alpha_{-2})^2, $  or $ {a_1}/{a_2} = (s - \alpha_{-2})/(s - \alpha_{-1}). 
$
Thus, in view of \eqref{a1},  we obtain
$
0= a_2 \alpha_{-2} - a_1 \alpha_{-1} = (a_2  - a_1) s <0, 
$
which is a contradiction.  
\end{proof}

Next, we will need an upper estimate for $x''(s)$ on the interval $[1,2].$ We have 
$$\max_{s \in [1, 2]}x''(s) = \max_{s \in [0, 1]}x''(s+1) =
\max_{s \in [0, 1]} \{f'(x(s)) \cdot f(x(s - 1))\}. 
$$

Since \( z_+(s - 1) \geq x(s - 1) \geq 0 \) for \( s \in [0, 1] \), 
we have that \( f(z_+(s - 1)) \leq f(x(s - 1)) \leq 0 \).
Now, by Lemma \ref{2.1},  since \( f(x) \geq r(x) \) for \( x \geq 0 \), it follows that \( f(z_+(s - 1)) \geq r(z_+(s - 1)) \). Thus
$
r(z_+(s - 1)) \leq f(x(s - 1)) \leq 0,$ $ s \in [0, 1].
$
By the same lemma, 
\( f'(x) \geq r'(x)=a/(1+x)^2 \) for all \( x \in [-1, 0] \). Since $r'(x)$ 
is increasing on \((-1, 0] \),  we obtain:
\begin{equation}\label{D}
\max_{s \in [1, 2]}x''(s) =\max_{s \in [0, 1]} \{f'(x(s)) \cdot f(x(s - 1))\} \leq \max_{s \in [0, 1]} \{r'(z_-(s)) \cdot r(z_+(s - 1))\}. 
\end{equation}

Defining \( z : \mathbb{R} \to \mathbb{R} \) as a combination of $z_+(t)$ and $z_-(t)$: 
$$
z(s) = z(s,M,m)=
\begin{cases}
M, & \text{if } s \in (-\infty, \alpha_+], \\
M + \frac{a r(m)}{2} (s - \alpha_+)^2, & \text{if } s \in [\alpha_+, \beta_+], \\
r(M) s, & \text{if } s \in [\beta_+, \beta_-], \\
m + \frac{a r(M)}{2}  (s - \alpha_-)^2, & \text{if } s \in [\beta_-, \alpha_-], \\
m, & \text{if } s \in [\alpha_-, \infty),
\end{cases}
$$
we can express the right-hand side of \eqref{D} as 
$$
D_{a}(M,m) := a^2 \max_{s \in [0,1]} \frac{z(s - 1, M, m)}{(1 + z(s, M, m))^2(1 + z(s - 1, M, m))}.
$$
Define also $
Z_{a}(M,m):= D_{a}(M,m)/a^2.$
These two functions, considered  in the domain  where $\beta_+ <0< \beta_-$ (this includes $
 \{(M, m): 0< M \leq A(m),\ -1 <m <0\}
$), play a key role in our analysis of the upper bound for \( M \). Below we proceed by establishing their monotonicity properties in the variables $a$, $M$ and $m$. 

A direct verification using formulas \eqref{20} and \eqref{39} shows that
$
m_1 > m_2$ implies 
$$
\alpha_{-,1} < \alpha_{-,2}, \
\beta_{-,1} < \beta_{-,2}, \quad \mbox{and}\quad
\alpha_{+,1} < \alpha_{+,2}, \
\beta_{+,2} < \beta_{+,1}
$$
(here index $j$ in $m_j$ is related to $j$ in $\alpha_{\pm,j}$ and $\beta_{\pm,j}$). 
\begin{lemma} \label{L3.12} Let $m_1 > m_2,$ then \(z(s-1, M, m_2) \geq z(s-1, M, m_1)\)  when \(s-1  \in [\alpha_{+,2}, \beta_{+,2}]\). Consequently, for \(s \in [0, 1)\),
\[
\frac{z(s - 1, M, m_1)}{1 + z(s - 1, M, m_1)} \leq \frac{z(s - 1, M, m_2)}{1 + z(s - 1, M, m_2)}.
\]
\end{lemma}
\begin{proof} 
Clearly \(z(s , M, m_2) = z(s , M, m_1)\) on \([ \beta_{+,1},0]\).
Suppose that  \(z(s,M, m_1)\) and \(z(s,M, m_2)\) intersect in the interval \([\alpha_{+,2}, \beta_{+,2}]\) at the rightmost point $\tilde s_2.$
Then we have $z(\tilde s_2, M, m_1) = z(\tilde s_2, M, m_2)$
and $z'(\tilde s_2, M, m_1) \leq z'(\tilde s_2, M, m_2)$. 
Using the explicit formulas for  $z(\tilde s_2, M, m_j)$, we obtain a contradiction: 
\begin{equation*}\label{45}
1> \frac{(\tilde s_2 - \alpha_{+,2})^2}{(\tilde s_2 - \alpha_{+,1})^2}=\frac{r(m_1)}{r(m_2)} \geq \frac{\tilde s_2 - \alpha_{+,2}}{\tilde s_2 - \alpha_{+,1}} >0. 
\end{equation*}
\end{proof}
\begin{lemma}  \label{L3.13}  Let $m_1 > m_2,$ if \( s \in [0, 1] \), then $z(s, M, m_1) \geq z(s, M, m_2).$ 
\end{lemma}
\begin{proof} 
It suffices to compare the parabolic segments of two functions on the interval \([ \beta_{-,2}, \alpha_{-,1}]\). Since the corresponding quadratic polynomials have the same leading coefficient, they can have at most one intersection point. This implies that  there is no intersection on the interval \([ \beta_{-,2}, \alpha_{-,1}]\)  (see the arguments in the proof of Lemma \ref{3.2}).
\end{proof}

\begin{lemma}\label{3.4}
Let \(-2 < a_2 < a_1 < 0 \), $0< M_1<M_2$ and $M_j \leq  A_{a_j}(m)$.  Then

a) \( D_{a_1}(M,m) \leq D_{a_2}(M,m) \) and 
 \(Z_{a_1}(M,m) \leq Z_{a_2}(M,m)\);
 
b) \(  0< D_{a_1}(M_1,m) < D_{a_1}(M_2,m) \leq a_1^2M_2/((1+m)^2(1+M_2))  \); 

c) if  $m_1 > m_2,$ then \(D_{a}(M, m_2) \geq D_{a}(M, m_1)\).
\end{lemma}

\begin{proof}
Lemmas \ref{2.6} and \ref{3.3} and the monotonicity properties of the function \( a^2x/[(1 + y)^2(1 + x)] \) imply  a).  Next, since $z(s,M_2,m) \leq z(s,M_1,m)$ for $s>0$ and $z(s,M_1,m) < z(s,M_2,m)$ for $s<0$, we easily deduce  b). Finally, the assertion in c) is a direct consequence of Lemmas \ref{L3.12}, \ref{L3.13} and  the definition of $D_{a}(M, m)$. 
\end{proof}

From now on, we  assume that  
$a \in [-37/24,-3/2]$ and that \( x(t) \) is a slowly oscillatory periodic solution with $\min x(s)=\bf{m}$.  By shifting the time,  we  will also assume that $x(1) =\max x(s)=\bf{M},$ and, consequently, $x(0)=0$. 
We are going to use  the following function with $m=\bf{m}$ and $M=\bf{M}$ as a lower bound for \( x(s) \) on  \([-1, 0] \): 
\[
\tilde{z}(s) =\tilde{z}(s, M, m)=
\begin{cases}
m, & \text{if } s \in (-\infty, \tilde{\alpha}], \\
m + \frac{D_{a}(M,m)}{2} (s - \tilde{\alpha})^2, & \text{if } s \in (\tilde{\alpha}, \tilde{\beta}], \\
r(m)s, & \text{if } s \in [\tilde{\beta}, 0],
\end{cases}
\]
where
\[
\tilde{\alpha} = \frac{m}{r(m)} - \frac{r(m)}{2D_a(M,m)} \quad \text{and} \quad \tilde{\beta} = \frac{m}{r(m)} + \frac{r(m)}{2D_a(M,m)}.
\]
{
\begin{lemma} \label{4.8}
If $a \in [-37/24,-3/2],$ then \( \tilde{\beta} < 0 \) for all $(M, m)$  in the sets
$${\frak A}_1=\{(M, m): -m\leq M \leq A_{-37/24}(m),\ m \in [-0.2343, -0.0093]\},$$
$${\frak A}_2= \{(M, m): 0.27\leq M \leq A_{-37/24}(m),\ m \in [-0.25, -0.21]\}.
$$
The union ${\frak A}= {\frak A}_1\cup {\frak A}_2$ of these sets contains $\mathcal{L}\left(-{37}/ {24}\right)$.
\end{lemma}
\begin{proof} 
First, consider  $(M,m) \in {\frak A}_1$. Since \( \tilde{\beta} < 0 \)
amounts to $Z_{a}(M,m) > -0.5m/(1+m)^2$ 
and  Lemma \ref{3.4} implies that $Z_{a}(M,m) \geq Z_{-3/2}(-m,m)$ for the indicated range of $M$, the inequality 
\begin{equation}\label{Z46}
-0.5m/(1+m)^2<  Z_{-3/2}(-m,m)
\end{equation}
ensures that \( \tilde{\beta} < 0 \) for $a \in [-37/24,-3/2]$. 
Applying Lemma \ref{BL} with 
$q(t)=Z_{-3/2}(t,-t)$, $p(t)= 0.5t/(1-t)^2$ and $[t_*,t^*] = [0.0093, 0.2343]$, we conclude that the inequality (\ref{Z46}) holds for all $m\in [-0.2343, -0.0093]$.  See the second MATLAB script in Appendix D, which yields $J_0=64$. Note that we are using a rigorous {\it lower} approximation for $a^2Z_a(M,m)$ presented in the first script of 
Appendix D.

Similarly, if $0.27 \leq M \leq A_a(m)$ and $m \in [-0.25, -0.21]$, then 
\begin{equation*}\label{Z}
-0.5m/(1+m)^2<  Z_{-3/2}(0.27,m) < Z_a(M,m).
\end{equation*}
Again, Lemma \ref{BL} with 
$q(t)=Z_{-3/2}(0.27,-t)$, $p(t)= 0.5t/(1-t)^2$ and $[t_*,t^*] = [0.21, 0.25]$ is used.  See the third script in Appendix D,    $J_0=13$.
This proves the second assertion of the lemma. 

Finally, from Theorem \ref{C12} and Remark \ref{R2.18}, we have that $-0.19 <m_{5}(0.27) < \hat L^{-1}(0.27)$ 
so that if $(M,m) \in \mathcal{L}\left(-{37}/ {24}\right)$  for $m \in [-0.25, -0.21]$ then $M\geq \hat L(-0.21)$ $> \hat L(-0.19) > 0.27$. If $(M,m) \in \mathcal{L}\left(-{37}/ {24}\right)$  with  $m \in [-0.21,-0.009]$ then $M \geq \hat L(m) \geq -m$ by Theorem \ref{C12}. 
\end{proof}
}
{
\begin{lemma}\label{L36} If $0<M_1 <M_2$ and $(M_j,m) \in  \mathcal{L}\left(-{37}/{24}\right)$, $j=1,2$, then it holds that 
$\tilde{z}(s,M_1) \not\equiv  
\tilde{z}(s,M_2)$ and \ 
$\tilde{z}(s,M_1) \geq  
\tilde{z}(s,M_2), \quad \text{for all }  s \in [-1,0].$
\end{lemma}
\begin{proof}
 This property is a straightforward consequence of Lemma \ref{3.4} and the definition of $\tilde{z}(s,M_j)$.
\end{proof}}
\begin{lemma}\label{L3.6}
\( x''(s) \leq D_{a}(\mathbf M, \mathbf  m) \) for all \( s \in [-1, 0] \) such that \( x'(s) \geq 0 \).
\end{lemma}
\begin{proof}
Due to the negative feedback condition,   \( x(s - 1) \leq 0 \) if \( x'(s) \geq 0 \).   Next,  we have that \( x''(s) = f'(x(s - 1)) f(x(s - 2)) \).   Since \( f' < 0 \), to obtain a positive upper bound for $x''(s)$, we need to consider \( x(s - 2) \geq 0 \). Therefore, in view of the sine-like shape of $x(t)$, the point \( s - 1 \) must belong to the decreasing negative segment of \( x(t) \). 
Using estimate \eqref{D}, which is valid on this segment, we complete the proof. 
\end{proof}

\begin{lemma} \label{L3.7}
\( x(s) \geq \tilde{z}(s) \) for all $s \in [-1,0].$
\end{lemma}
\begin{proof}
For all values of \( s \in [-1, 0] \) corresponding to the increasing segment of $x(s)$, we have from Lemma \ref{2.1} that 
\[
x(s) = - \int^0_s f(x(u - 1)) \, du \geq - \int^0_s r(x(u - 1)) \, du \geq - \int^0_s r(\mathbf m) \, du = r(\mathbf  m) s.
\]
Thus, \( x(s) \geq \max_{s \in [-1,0]} \{r(\mathbf  m) s, \mathbf  m\} \). Now, suppose that  \( x(s_1) < \tilde{z}(s_1) \) for some  \( s_1 \in [\tilde{\alpha}, \tilde{\beta}] \).
Then  there exist \(a\) and \(b\),  \(a < b \leq  \tilde{\beta} \),  such that \( x'(a) \leq \tilde{z}'(a) \), \( x'(b) \geq \tilde{z}'(b) \), \( x(b) = \tilde{z}(b) \), and \( x(s) \) is increasing on \( (a, b) \). By Lemma \ref{L3.6} and the definition of  $\tilde z(s)$, we know that \( x''(s) -\tilde{z}''(s)\leq 0 \) for \( s \in [a, b] \). Therefore, integrating this inequality from $a$ to $b$ we get 
$
0 \leq x'(b) - \tilde{z}'(b) \leq x'(a) - \tilde{z}'(a) \leq 0.
$
Hence, 
\[
\int_a^b x''(s) \, ds = \int_a^b \tilde{z}''(s) \, ds,
\]
and this, together with Lemma \ref{L3.6},  implies that  \( x''(s)  = \tilde{z}''(s) \) for \( s \in [a, b] \) and \( x'(b) = \tilde{z}'(b) \), \( x(b) = \tilde{z}(b) \). Consequently, we have \( x(s_1) = \tilde{z}(s_1) \), a contradiction.
\end{proof}

For $(M,m) \in \frak A$, set
\[
\Sigma_a(M, m) = \int_{-1}^0 r(\tilde{z}(s)) \, ds =  \int^0_{-1} \frac{a \tilde{z}(s, M, m)}{1 + \tilde{z}(s, M, m)} ds.
\]
\begin{corollary}\label{C3.8aa} \(\mathbf{M} \leq \Sigma_{a}(\bf{M}, \bf{m}).\)  
\end{corollary}
\begin{proof}
Indeed, by applying  Lemma \ref{L3.7},  \ref{2.1}, along with the monotonicity of the function $r$, we obtain  that 
\begin{equation}\label{M1}
{\bf M} = \int_0^1 f(x(s - 1)) \, ds = \int_{-1}^0 f(x(s)) \, ds \leq \int_{-1}^0 r(x(s)) \;ds\leq \int_{-1}^0 r(\tilde{z}(s)) \, ds.
\end{equation}
\end{proof}
\begin{lemma}\label{L3.9} If \(a_2 < a_1 < 0\), then \(\tilde{z}_{a_2}(s) \leq \tilde{z}_{a_1}(s)\) for \(s \in [-1, 0]\).
\end{lemma}
\begin{proof} By Lemma \ref{3.4}, we have \(Z_{a_2} \geq Z_{a_1} > 0\). Since 
\[
\tilde{\alpha}_{a_j} = \frac{m}{r_{a_j}(m)} - \frac{r_{a_j}(m)}{2D_{a}(M,m)} = \frac{1}{a_j} \left( 1 + m - \frac{m}{2(1+m)Z_{a_j}} \right),
\]
we obtain
$
{\tilde{\alpha}_{a_2}}/{\tilde{\alpha}_{a_1}} \leq {a_1}/{a_2} < 1$ and therefore  $\tilde{\alpha}_{a_2} > \tilde{\alpha}_{a_1}$. 

 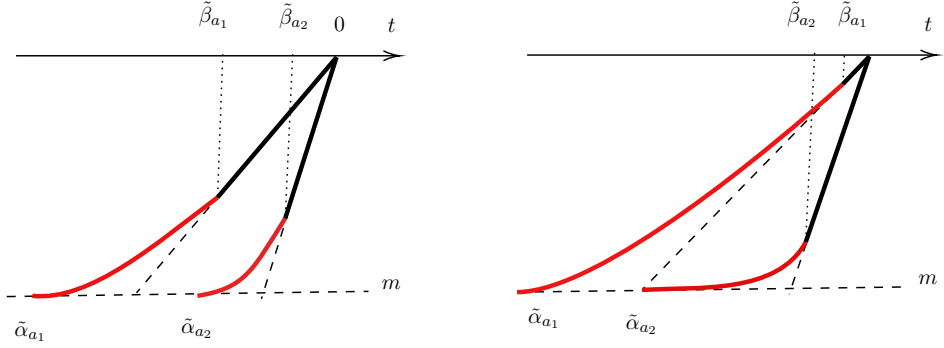
\begin{figure}[h] \label{F3}
\centering
\scalebox{0.75}
{

\tikzset{every picture/.style={line width=0.75pt}} 

\begin{tikzpicture}[x=0.75pt,y=0.75pt,yscale=-1,xscale=1]

\draw    (11.5,76) -- (267.5,76) ;
\draw [shift={(269.5,76)}, rotate = 180] [color={rgb, 255:red, 0; green, 0; blue, 0 }  ][line width=0.75]    (10.93,-3.29) .. controls (6.95,-1.4) and (3.31,-0.3) .. (0,0) .. controls (3.31,0.3) and (6.95,1.4) .. (10.93,3.29)   ;
\draw  [dash pattern={on 4.5pt off 4.5pt}]  (5.25,237.5) -- (248.5,234) ;
\draw  [dash pattern={on 4.5pt off 4.5pt}]  (227,77) -- (176.5,239) ;
\draw  [dash pattern={on 4.5pt off 4.5pt}]  (227,77) -- (92.5,235) ;
\draw [color={rgb, 255:red, 238; green, 17; blue, 17 }  ,draw opacity=1 ][line width=2.25]    (22.5,237) .. controls (61.5,241) and (107,201) .. (147,171) ;
\draw [line width=2.25]    (227,77) -- (147,171) ;
\draw [color={rgb, 255:red, 228; green, 35; blue, 35 }  ,draw opacity=1 ][line width=2.25]    (133.5,237) .. controls (165.5,232) and (170.5,220) .. (192.5,185) ;
\draw [line width=2.25]    (227,77) -- (192.5,185) ;
\draw  [dash pattern={on 0.84pt off 2.51pt}]  (150.5,75) -- (147,171) ;
\draw  [dash pattern={on 0.84pt off 2.51pt}]  (197.5,75) -- (192.5,185) ;
\draw    (355.17,75.52) -- (628.65,75.52) ;
\draw [shift={(630.65,75.52)}, rotate = 180] [color={rgb, 255:red, 0; green, 0; blue, 0 }  ][line width=0.75]    (10.93,-3.29) .. controls (6.95,-1.4) and (3.31,-0.3) .. (0,0) .. controls (3.31,0.3) and (6.95,1.4) .. (10.93,3.29)   ;
\draw  [dash pattern={on 4.5pt off 4.5pt}]  (348.5,234.52) -- (608.23,231.07) ;
\draw  [dash pattern={on 4.5pt off 4.5pt}]  (585.27,76.5) -- (531.35,235.99) ;
\draw  [dash pattern={on 4.5pt off 4.5pt}]  (585.27,76.5) -- (432.5,233) ;
\draw [color={rgb, 255:red, 240; green, 20; blue, 20 }  ,draw opacity=1 ][line width=2.25]    (348.5,234.52) .. controls (404.5,234) and (519.5,138) .. (567.5,95) ;
\draw [line width=2.25]    (585.27,76.5) -- (567.5,95) ;
\draw [color={rgb, 255:red, 228; green, 11; blue, 11 }  ,draw opacity=1 ][line width=2.25]    (432.5,233) .. controls (470.5,231) and (523.5,236) .. (542.5,201) ;
\draw [line width=2.25]    (585.27,76.5) -- (542.5,201) ;
\draw  [dash pattern={on 0.84pt off 2.51pt}]  (568.5,76) -- (567.5,95) ;
\draw  [dash pattern={on 0.84pt off 2.51pt}]  (548.5,74) -- (542.5,201) ;

\draw (9,252.4) node [anchor=north west][inner sep=0.75pt]    {$\tilde{\alpha }_{a_{1}}$};
\draw (121,250.4) node [anchor=north west][inner sep=0.75pt]    {$\tilde{\alpha }_{a_{2}}$};
\draw (133,38.4) node [anchor=north west][inner sep=0.75pt]    {$\tilde{\beta }_{a_{1}}$};
\draw (187,39.4) node [anchor=north west][inner sep=0.75pt]    {$\tilde{\beta }_{a_{2}}$};
\draw (256,224.4) node [anchor=north west][inner sep=0.75pt]    {$m$};
\draw (260,49.4) node [anchor=north west][inner sep=0.75pt]    {$t$};
\draw (224,48.4) node [anchor=north west][inner sep=0.75pt]    {$0$};
\draw (354.45,243.95) node [anchor=north west][inner sep=0.75pt]    {$\tilde{\alpha }_{a_{1}}$};
\draw (420.04,247.98) node [anchor=north west][inner sep=0.75pt]    {$\tilde{\alpha }_{a_{2}}$};
\draw (562.82,38.24) node [anchor=north west][inner sep=0.75pt]    {$\tilde{\beta }_{a_{1}}$};
\draw (528.48,38.23) node [anchor=north west][inner sep=0.75pt]    {$\tilde{\beta }_{a_{2}}$};
\draw (616.75,221.5) node [anchor=north west][inner sep=0.75pt]    {$m$};
\draw (620.92,49.21) node [anchor=north west][inner sep=0.75pt]    {$t$};

\end{tikzpicture}

}
\vspace{-0mm}

\caption{\hspace{-2mm} Two possible positions of the graphs of  $\tilde{z}_{a_2}(s)$, $\tilde{z}_{a_1}(s)$, with parabolic parts shown in red.}
\end{figure}

It is clear that in \([\tilde{\alpha}_{a_1}, \tilde{\alpha}_{a_2}] \) and \([\max_j \tilde{\beta}_{a_j}, 0] \), the conclusion of the lemma is true. 
Suppose that \(\tilde{\alpha}_{a_2} <\min_j\tilde{\beta}_{a_j}\) and there is an intersection of the graphs of $\tilde{z}_{a_2}(s)$ and $ \tilde{z}_{a_1}(s)$ at some point \(s_1 \in [\tilde{\alpha}_{a_2}, \min_j\tilde{\beta}_{a_j} ]\), i.e.  $\tilde{z}_{a_2}(s_1) = \tilde{z}_{a_1}(s_1)$ and  
\begin{equation}\label{DD}
 D_{a_1}(M,m)(s - \tilde{\alpha}_{a_1})^2 = D_{a_2}(M,m)(s -\tilde{\alpha}_{a_2})^2, \quad s=s_1.     
\end{equation}
Consider the parabolas in the definition of  $\tilde{z}_{a}(s)$  corresponding to  \(a= a_1\) and \(a= a_2\). Since \(0< D_{a_1} < D_{a_2}\) (by Lemma  \ref{3.4}), they have exactly two intersections, one at \(s^\ast \in (\tilde{\alpha}_{a_1}, \tilde{\alpha}_{a_2})\) ({which is irrelevant for our argument}) and another at \(s_1 > \tilde{\alpha}_{a_2}\), cf. equation (\ref{DD}). This implies that the parabola corresponding to the parameter \(a_2\) will never intersect the inclined linear part, \(r_{a_1 }(m) s\), of \(z_{a_1}\) and, therefore, will not intersect the line \( r_{a_2}(m) s\), a contradiction. 
Thus there is no intersection of the  parabolas in \([ \tilde{\alpha}_{a_2}, \min_j\tilde{\beta}_{a_j}]\).

Finally, the graphs of $\tilde{z}_{a_2}(s)$ and $ \tilde{z}_{a_1}(s)$ on \([\min_j \tilde{\beta}_{a_j}, \max_j\tilde{\beta}_{a_j}]\) 
are  given by a segment of a straight line and by an arc of a parabola. An easy  analysis shows that they cannot intersect  because of the convexity of the parabola, see Fig. 10.
\end{proof}

Define now
$
\mathcal{U}(a) = \{(M,m) \in {\frak A} : M \leq \Sigma_a(M,m)\}.$
\begin{lemma}
\(\mathcal{U}(a_1) \subseteq \mathcal{U}(a_2)\), if \(0 > a_1 \geq a_2\).
\end{lemma} 
\begin{proof} 
By Lemma \ref{L3.9}, we have that \(\tilde{z}_{a_2}(s) \leq \tilde{z}_{a_1}(s) < 0\), hence
\[
r_{a_1} (\tilde{z}_{a_1}(s)) = \frac{a_1 \tilde{z}_{a_1}(s)}{1 + \tilde{z}_{a_1}(s)} \leq \frac{a_2 \tilde{z}_{a_1}(s)}{1 + \tilde{z}_{a_1}(s)}  \leq \frac{a_2\tilde{z}_{a_2}(s)}{1 + \tilde{z}_{a_2}(s)}  = r_{a_2}(\tilde{z}_{a_2}(s)).
\]
\begin{equation*}\label{42}
\hspace{-17mm} \mbox{Thus,} \quad \Sigma_{{a_1}}(M,m) = \int^0_{-1} r_{a_1} (\tilde{z}_{a_1}(s)) ds \leq \int^0_{-1} r_{a_2} (\tilde{z}_{a_2}(s)) ds = \Sigma_{{a_2}}(M,m),
\end{equation*}
which implies $\mathcal{U}(a_1) \subseteq \mathcal{U}(a_2).$
\end{proof}
\begin{lemma}\label{L3.16}
If \( m_1 > m_2 \), then
$
\Sigma_a(M, m_1) < \Sigma_a(M, m_2).
$
\end{lemma} 
\begin{proof} 
Since \( \tilde{z}(s, M, m_1) \not\equiv \tilde{z}(s, M, m_2) \) for \( m_1 \neq m_2 \), it is sufficient to show that it holds \( \tilde{z}(s, M, m_2) \leq \tilde{z}(s, M, m_1) \) for all \( s \in [-1, 0] \).

The only possible  intersections of the graphs of \( \tilde{z}(s, M, m_2) \) and \( \tilde{z}(s, M, m_1) \) are if the parabolic part of \( \tilde{z}(s, M, m_2) \) intersects \( \tilde{z}(s, M, m_1) \).

Clearly, the parabolic parts of \( \tilde{z}(s, M, m_1) \) and \( \tilde{z}(s, M, m_2) \) can have at most two intersections. Moreover, since \( D_{a}(M, m_2) \geq D_{a}(M, m_1) \), they can have at most one transversal intersection at some point \( s^* >\tilde{\alpha}_2.\) 

Regardless of whether $s^*$ corresponds to the intersection of the two parabolic segments or to the intersection of the parabolic part of \( \tilde{z}(s, M, m_2) \) and the line segment of \( \tilde{z}(s, M, m_1)\), 
we easily obtain that  \( \tilde{z}(s, M, m_2) > \tilde{z}(s, M, m_1) \) for \( s \in (s^*, 0] \), which is impossible since \( r(m_2)s < r(m_1)s \) for \( s < 0 \). Therefore, we have a contradiction.

\end{proof}
%
\begin{lemma}\label{C3.8a}  Assume that  $0<M_1 <M_2$ and $(M_j,m) \in  \frak A$, for $j=1,2,$ then 
 $\Sigma_{a}(M_1, m) < \Sigma_{a}(M_2, m)$ \ and \  $\Sigma_{a}(M_1, m) < A_a(m)$. 
\end{lemma}
\begin{proof}
The monotonicity property of $\Sigma_a$ is ensured by Lemma \ref{L36}. Next, note that $\tilde{z}(s, M, m) \geq \tilde{z}_\infty(s,m)$, where the latter function is defined by the same formula as $\tilde{z}(s, M, m)$ with $D_a=+\infty$ (so that $\alpha_-=\beta_-=m/r(m) \in (-1,0)$,  the parabolic segment is eliminated and $\tilde{z}(s, M, m) \not= \tilde{z}_\infty(s,m)$). Thus 
$\Sigma_a(M, m) = \int^0_{-1}r(\tilde{z}(s, M, m))ds < \int^0_{-1}r(\tilde{z}_\infty(s,m))ds=A_a(m)$. 
\end{proof}

Lemmas \ref{C3.8a}, \ref{L3.16} imply that the  function $\theta_1$ defined by $\theta_1(m):= \Sigma_{-37/24}(A_a(m), m) < A_{-37/24}(m)$    is strictly decreasing and continuous on $[-0.25, -0.0093]$. 
\begin{lemma}\label{L4.18}
Consider the equation $M=\Sigma_{-37/24}(M, m)$ in the domain $\frak A$.  It has one maximal (i.e. the biggest) solution $M=\theta(m)$
for each $m \in [-0.25,-0.0093]$.
\end{lemma}

\begin{proof} For a fixed $m \in [-0.25,-0.0093]$, we will apply
the intermediate value theorem  to the function $\gamma(M):= M-\Sigma_{-37/24}(M, m)$.  Indeed, as we have just proved, $\gamma(A_{-37/24}(m))>0$. 

Next,  for each fixed $m\in [-0.25, -0.2343]$,  we have $\gamma(0.27)<0$ since  $$\Sigma_{-37/24}(0.27, m) \geq \Sigma_{-37/24}(0.27, -0.2343) = 0.30... >0.27.$$

Now, if $m \in [-0.22,-0.0093]$, we 
invoke Lemma \ref{BL} with  $q(t)=  \Sigma_{-37/24}(t, -t),$ $ p(t)=t$ and $[t_*,t^*]= [0.0093, 0.26]$, $J_0=97,$ to conclude that $\Sigma_{-37/24}(-m, m) >-m$ for  $m\in [-0.22,-0.0093]$. See the second MATLAB script in Appendix E.  Note that we are using a rigorous {\it lower} approximation for $\Sigma_a(M,m)$ presented in the first script of the same appendix. As a consequence,  $\gamma(-m) <0$.
\end{proof}

The following obvious relation is used in the proof of Theorem \ref{T3.18} below:
\begin{equation} \label{MS}
M>\Sigma_{-37/24}(M, m) \ \mbox{whenever} \ M \in (\theta(m), A_{-37/24}(m)). 
\end{equation} 

\vspace{-2mm}

\begin{theorem} \label{T3.18} (i)
The recursive  sequence
$\theta_n(m)$, $m \in \frak I= [-0.25, -0.0093]$, 
\begin{equation*}\label{theta}
\theta_{n+1}(m)= \Sigma_{-37/24}(\theta_n(m), m), \ n=0,1,2\dots, \quad  \theta_0(m)=A_{-37/24}(m),
\end{equation*}
is monotonically decreasing and consists of  decreasing continuous functions converging to $\theta(m)$. (ii) Each admissible pair $({\mathbf M, \mathbf m})$ satisfies $\mathbf M < \theta_n(\mathbf m)$ for all $n \in \N$.
\end{theorem}
\begin{proof}
As we know, $\theta_{1}(m) <\theta_0(m)$ and,    assuming that $\theta_{n}(m) <\theta_{n-1}(m)$ on $\frak I$, we find 
that 
$
\theta_{n+1}(m)= \Sigma_{-37/24}(\theta_n(m), m) < \Sigma_{-37/24}(\theta_{n-1}(m), m)=\theta_{n}(m), \ n \in \N.  
$
Clearly, the monotonicity properties of $\Sigma_a(M,m)$ imply that $\Sigma_{-37/24}(\theta_n(m), m)$ is strictly decreasing if $\theta_n(m)$ is strictly decreasing. 
The application of mathematical induction completes the proof of (i). 

(ii) In view of Corollary \ref{C3.8aa} and (\ref{MS}), we get
$\mathbf M <  \theta(\mathbf m) <\theta_n(\mathbf m).$
\end{proof}

\section{ Proof of Theorem \ref{1.4}}
\label{Sec4}

By Theorems \ref{C12} and \ref{T3.18}, each admissible pair $({\mathbf M, \mathbf m})$  satisfies the relations \( \mathbf m \in [-0.25, -0.0093] \), $ \mathbf M\in [0.009401, 0.377]$ and $$\mathbf m\geq m_k(\mathbf M), \quad \mathbf M < \theta_n(\mathbf m) \quad \mbox {for all} \  n, k \in \N.$$  The function $m_k(M)$ is well defined on $ [0.009401, 0.377]$,  strictly decreasing and continuous on this interval. See the blue curve in Fig. 11 for a schematic representation of $m_k$. Note that each admissible pair $({\mathbf M, \mathbf m})$  lies above this curve. 

\begin{figure}[h] \label{F3}
\centering
\scalebox{0.75}
{
\tikzset{every picture/.style={line width=0.75pt}} 

\begin{tikzpicture}[x=0.75pt,y=0.75pt,yscale=-1,xscale=1]

\draw    (187.33,99.17) -- (517.33,100.16) ;
\draw [shift={(519.33,100.17)}, rotate = 180.17] [color={rgb, 255:red, 0; green, 0; blue, 0 }  ][line width=0.75]    (10.93,-3.29) .. controls (6.95,-1.4) and (3.31,-0.3) .. (0,0) .. controls (3.31,0.3) and (6.95,1.4) .. (10.93,3.29)   ;
\draw    (237.5,296.5) -- (237.5,54.5) ;
\draw [shift={(237.5,52.5)}, rotate = 90] [color={rgb, 255:red, 0; green, 0; blue, 0 }  ][line width=0.75]    (10.93,-3.29) .. controls (6.95,-1.4) and (3.31,-0.3) .. (0,0) .. controls (3.31,0.3) and (6.95,1.4) .. (10.93,3.29)   ;
\draw [line width=0.75]  [dash pattern={on 4.5pt off 4.5pt}]  (479.67,97.18) -- (476.67,288.18) ;
\draw  [dash pattern={on 4.5pt off 4.5pt}]  (239,279) -- (477.67,282.18) ;
\draw [line width=2.25]  [dash pattern={on 6.75pt off 4.5pt}]  (238.5,133.5) -- (478.67,133.18) ;
\draw [color={rgb, 255:red, 74; green, 144; blue, 226 }  ,draw opacity=1 ][line width=2.25]    (239,102) .. controls (283.67,165.18) and (435.17,186.68) .. (478.17,192.68) ;
\draw [color={rgb, 255:red, 219; green, 30; blue, 30 }  ,draw opacity=1 ][line width=3]    (239,102) .. controls (267.67,149.18) and (361.67,254.18) .. (441.67,281.18) ;
\draw  [dash pattern={on 0.84pt off 2.51pt}]  (236.67,192.18) -- (478.17,192.68) ;
\draw  [dash pattern={on 0.84pt off 2.51pt}]  (311.67,101.17) -- (312,189) ;
\draw  [dash pattern={on 0.84pt off 2.51pt}]  (236.67,152.17) -- (310.67,153.17) ;
\draw  [dash pattern={on 0.84pt off 2.51pt}]  (279.67,101.17) -- (279.67,152.67) ;
\draw  [dash pattern={on 0.84pt off 2.51pt}]  (239.33,138.67) -- (280.67,139.17) ;
\draw  [dash pattern={on 0.84pt off 2.51pt}]  (265.67,100.17) -- (265.67,139.17) ;
\draw  [dash pattern={on 0.84pt off 2.51pt}]  (236.67,127.46) -- (264,128) ;

\draw (214,73.73) node [anchor=north west][inner sep=0.75pt]    {$0$};
\draw (525,73.07) node [anchor=north west][inner sep=0.75pt]    {$M$};
\draw (206,288.73) node [anchor=north west][inner sep=0.75pt]    {$m$};
\draw (469,74.73) node [anchor=north west][inner sep=0.75pt]    {$M_{0}$};
\draw (303,77.73) node [anchor=north west][inner sep=0.75pt]    {$M_{1}$};
\draw (276,76.73) node [anchor=north west][inner sep=0.75pt]    {$M_{2}$};
\draw (248,76.73) node [anchor=north west][inner sep=0.75pt]    {$M_{3}$};
\draw (205,188.73) node [anchor=north west][inner sep=0.75pt]    {$m_{0}$};
\draw (205,147.34) node [anchor=north west][inner sep=0.75pt]  [rotate=-1.62]  {$m_{1}$};
\draw (205,133.73) node [anchor=north west][inner sep=0.75pt]    {$m_{2}$};
\draw (205,123.07) node [anchor=north west][inner sep=0.75pt]    {$m_{3}$};
\draw (484,122.73) node [anchor=north west][inner sep=0.75pt]    {$-0.0093$};
\draw (385,288.73) node [anchor=north west][inner sep=0.75pt]  [color={rgb, 255:red, 208; green, 2; blue, 27 }  ,opacity=1 ]  {$M=\theta _{n}( m)$};
\draw (485,182.73) node [anchor=north west][inner sep=0.75pt]  [color={rgb, 255:red, 74; green, 144; blue, 226 }  ,opacity=1 ]  {$m=m_{k}( M)$};
\draw (474,94.73) node [anchor=north west][inner sep=0.75pt]    {$\bullet $};
\draw (306,94.73) node [anchor=north west][inner sep=0.75pt]    {$\bullet $};
\draw (275,94.73) node [anchor=north west][inner sep=0.75pt]    {$\bullet $};
\draw (260,94.73) node [anchor=north west][inner sep=0.75pt]    {$\bullet $};

\draw (474,188) node [anchor=north west][inner sep=0.75pt]    {\Large $\bullet $};
\draw (308,188) node [anchor=north west][inner sep=0.75pt]    {\Large $\bullet $};
\draw (308,149) node [anchor=north west][inner sep=0.75pt]    {\Large $\bullet $};
\draw (274,150) node [anchor=north west][inner sep=0.75pt]    {\Large $\bullet $};
\draw (275,134) node [anchor=north west][inner sep=0.75pt]    {\Large $\bullet $};
\draw (261,134) node [anchor=north west][inner sep=0.75pt]    {\Large $\bullet $};
\draw (261,123) node [anchor=north west][inner sep=0.75pt]  [color={rgb, 255:red, 216; green, 13; blue, 13 }  ,opacity=1 ]  {\Large $\bullet $};
\end{tikzpicture}

}
\vspace{-0mm}

\caption{\hspace{0cm} Schematic diagram with the graphs of the functions $m_k(M)$ (blue)  and $\theta_n(m)$ (red), and the path of the 'billiard ball'. }
\end{figure}
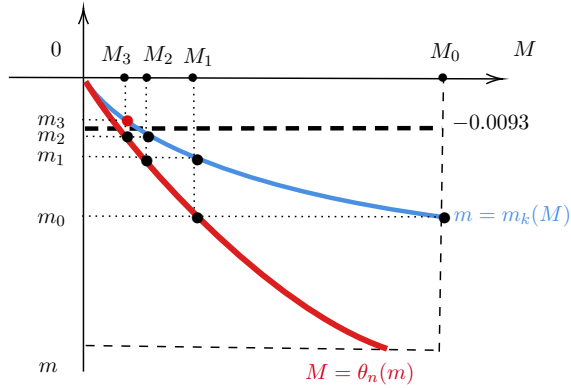 

Similarly,  the red curve in Fig. 11 represents the graph of  $\theta_n(\mathbf m)$: each admissible pair $({\mathbf M, \mathbf m})$  is on the left side of this curve. 

Consequently, if the mutual position of the graphs of $m_k$ and $\theta_n$ is as in Fig. 11, there are no admissible pairs $({\mathbf M, \mathbf m})$ and $
\mathcal{U}(-37/24) \cap \mathcal{L}(-37/24) = \emptyset.
$
This means that there are no slowly oscillating periodic solutions for \( a \in [-37/24,-3/2] \) which implies the global stability of the trivial equilibrium of \eqref{1.4}.

To show that the graphs of $m_k$ and $\theta_n$ for some suitable $n, k \in \N$ are located as in Fig.11, we will slightly modify the "billiard" test described in Section \ref{SecA}.   
Set $M_0=0.377$ and suppose that $M_1=\rho(M_0):=\theta_n(m_k(M_0)) < M_0,$ see Fig. 11. Then, for each $M\in [M_1,M_0]$, the graph of $m_k$
has exactly one intersection with the horizontal line $m=m_0=:m_k(M_0)$ at the point $(M_0,m_0)$. Similarly, the graph of $\theta_n$
intersects the horizontal line $m=m_0$ exactly once at the point $(M_1,m_0)$. This means that the graph of $m_k$ is located above the graph of $\theta_n$ in the vertical strip $[M_1,M_0]\times \R_-$. We can repeat this procedure while the sequence $M_j$ defined as 
$
M_{j}= \rho^j(M_0)
$ is strictly decreasing. Furthermore, if for some $j_0$ we find that
$m_{j_0}=: m_k(M_{j_0}) > -0.0093$ (see the red ball at the end of the sequence of black balls in Fig. 11 leaving the "billiard table" created by the black dashed horizontal line and the blue and red curves), we may conclude that the graphs of $m_k$ and $\theta_n$ do not intersect when
$-0.25 \leq m \leq  -0.0093$. Note that the curves meet at the origin, where they share the same tangent line \( m + M = 0 \)  for the critical value \( f'(0) = -37/24 \).

The above algorithm for the separation of the domains  $
\mathcal{U}(-37/24)$ and $ \mathcal{L}(-37/24)$  can be applied on each sub-interval $[M_*,M^*]$ of $[0.009401, 0.377]$. To verify that the graph of $m_k$ is above the graph of $\theta_n$ in the vertical strip $[M_*,M^*]\times \R_-$, it suffices to check whether the finite sequence 
$
M_{j}= \rho^j(M^*)
$ $j=1,\dots ,j^*$ is strictly decreasing and $M_{j^*} < M_*$. 
By applying verified computation
techniques, we executed  this work on a 16-core MacBook Pro 16/M4 Max computer for  sub-intervals of $[0.009401, 0.377]$ and also for the whole interval (see Table 1). 
\begin{table}[]
    \centering
   \begin{tabular}{| m{4cm}|c|c|}  \hline  
    \vspace{4.5mm}  \centerline{$M$-Interval}  &  \hbox{Number of Iterations} & \hbox{Elapsed Time} \\ \hline 
   $[0.1005, 0.377]$    & $63$ & $\ 13$\ sec \\ \hline
       $ [0.05004, 0.1005]$    & $209$ & $45$\ sec\\ \hline
     $ [0.04001, 0.05004]$    & $166$ & $37$\ sec\\ \hline
     $ [0.03001, 0.04001]$    & $381$ & $96$\ sec\\ \hline
     $ [0.020002, 0.03001]$    & $1274$ & $285$\ sec\\ 
        \hline
       {  $ [0.00940007, 0.377]$ with $m_{24336} =
  -0.009295\dots $ }  &  \hbox{ $24336$  } &  \hbox{$5533$\ sec  }\\ \hline
    \end{tabular}

\vspace{3mm}
    
    \caption{Separation of domains on  sub-intervals of $[0.00940007, 0.377]$. }
    \label{tab:my_label}

\vspace{-5mm}

\end{table}
This table also presents the number of iterations (steps) $j_*$ for each such interval. Note that $j_* $ depends on the numbers $n, k$ in $\theta_n$ and $m_k$. Moreover,  verified proofs must use the upper estimate for  $\theta_n(m)$ and the lower estimate for $m_k(M)$, the number $j_*$ depends also on the choice of the respective approximation algorithms. In Appendix F we present four MATLAB/INTLAB scripts used in this section. The first and the second scripts compute the upper bound for $D_a(M,m)$ and $\Sigma_a(M,m)$, respectively.  The third and fourth scripts allow to find the upper bound for $\theta_6(m)$ and the lower bound for $m_5(M)$,  the fifth script contains the iteration procedure. 
Clearly, as Fig. 1 and Table 1 show, the step $M_{j+1}-M_j$ is decreasing when  $j$ grows:
in fact, the last step before leaving the "billiard table" is of order  $0.5\cdot 10^{-7}$.



\section{Appendix 1: For $a=-37/24$ the curves $m= L_a(M,m)$ and  $M= \Sigma _a(M,m)$ are tangent at the origin}

\subsection{Local analysis of $m= L_{-37/24}(M,m)$}
To understand the behavior of  the function \( M = \hat{L}(m) \) near the origin, we will introduce a new variable $\mu =\frac{M}{m}$ instead of $m$. Expressing \( \alpha_+ \), \( \beta_+ \) in terms of \( \mu, M \), we obtain:
\[
m = \frac{M}{\mu}, \quad \beta_{+}(M,\mu) = \frac{1 + M}{a} +   \frac{\mu\left(\frac{M}{\mu}+1\right)}{2a(1 + M)} = \frac{1 + M}{a} + \frac{M+\mu}{2a(1 + M)}  ,
\]
\[
\alpha_{+}(M,\mu) = \frac{1 + M}{a} - \frac{M + \mu}{2a(M + 1)}.
\]
Taking the limits in the above expressions  as $M\to 0$ and $\mu \to -1$,   we find that 
\[
-1 < \alpha_{+0}:= \alpha_{+}(0,-1)= \frac{3}{2a} < \beta_{+0}: =\beta_{+}(0,-1) = \frac{1}{2a}, \quad a =-\frac{37}{24}.
\]
Therefore, for (\(M,\mu) \) close to \( (0, -1) \), we have
\begin{equation}\label{28}
L_a(M, m) = \int^0_{\beta_+} r(r(M) s) \, ds + \int_{\alpha_+}^{\beta_+} r(z_+(s)) \, ds + \int^{\alpha_+}_{-1} r(M) \, ds.
\end{equation}
Consequently, the equation \( m = L_a(M, m) \) has the following form in the new variables:
\[
M = \mu  L_a\left(M, \frac{M}{\mu}\right) = \mu  \int^0_{\beta_{+}} r(r(M) s) \, ds + \mu \int_{\alpha_{+}}^{\beta_{+}} r(z_+(s)) \, ds + \mu  r(M)(\alpha_+ + 1).
\]
In other words,  \( G(M, \mu) = 0 \), where
\begin{equation*}\label{29}
G(M, \mu)= -\frac{1}{\mu} + \int^0_{\beta_{+}} \frac{\frac{a^2 s}{1 + M}}{1 + r(M) s} \, ds + \int_{\alpha_{+}}^{\beta_{+}} 
\frac{a\left(1 + \frac{a r(m)}{2M} (s - \alpha_+)^2\right)}{1 + \left(M+\frac{a r(m)}{2} (s - \alpha_+)^2\right)} \, ds +  \frac{a(\alpha_+ + 1)}{1+M}.
\end{equation*}
Note that $G(M, \mu)$ is a   smooth function of \( \mu \) and \( M \) in a neighborhood of \( \mu = -1 \) and \( M = 0 \). Moreover,  we find that \( G(0, -1) = 0 \) : 
$$
 -1 = \int^0_{\beta_{+0}} a^2 s\;ds + \int_{\alpha_{+0}}^{\beta_{+0}} 
{a\left(1 -\frac{a^2 }{2} (s - \alpha_{+0})^2\right)}ds +  a (\alpha_{+0} + 1), \
\alpha_{+0}= \frac{3}{2a}, \beta_{+0}=  \frac{1}{2a}. 
$$

It is worth  noting here that the value $a= - {37}/{24}$ is uniquely determined by the latter relation. 

{In the subsequent computations we assume $a= - {37}/{24}$ even if we write for simplicity only $a.$} Using the following notation
\[
D_0(M, \mu) :=  \frac{a(1 + \frac{a r(m)}{2M}  (s - \alpha_+)^2)}{1 + M +\frac{a r(m)}{2}  (s - \alpha_+)^2} =: \frac{D_1(M, \mu) }{D_2(M, \mu)},
\]
we proceed further by  calculating the partial derivative \( G_\mu \) at \( (0, -1) \):
\[
\frac{\partial G}{\partial \mu} (0, -1) = 1 - \frac{\partial \beta_{+}}{\partial \mu}  a^2  \beta_{+0} + \frac{\partial \beta_{+}}{\partial \mu}  a \left(1 - \frac{a^2}{2} \left(\beta_{+0} - \alpha_{+0}\right)^2\right)
\]
\[
- \frac{\partial \alpha_{+}}{\partial \mu}  a + \int^{\frac{1}{2a}}_{\frac{3}{2a}} \frac{\partial D_0}{\partial \mu} (0, -1) \, ds + a  \frac{\partial \alpha_{+}}{\partial \mu}= 1 + \int^{\frac{1}{2a}}_{\frac{3}{2a}}  \frac{\partial D_0}{\partial \mu} (0, -1) \, ds.
\]
Now, since 
\[
D_1(0, -1) = a \left(1 - \frac{a^2}{2}  \left(s - \frac{3}{2a}\right)^2\right), \qquad
D_2(0, -1) = 1,
\]
\[
\frac{\partial D_1}{\partial \mu}(0, -1) = \frac{a^2}{2} \left(-a \left(s - \frac{3}{2a}\right)^2 - \left(s - \frac{3}{2a}\right)\right), \qquad
\frac{\partial D_2}{\partial \mu}(0, -1) = 0,
\]
we have
\[
\frac{\partial D_0}{\partial \mu}(0, -1) = \frac{D_2(0, -1) \cdot \frac{\partial D_1}{\partial \mu}(0, -1) - D_1(0, -1) \cdot \frac{\partial D_2}{\partial \mu}(0, -1)}{(D_2(0, -1))^2} = \frac{\partial D_1}{\partial \mu}(0, -1).
\]
Therefore,
\[
\frac{\partial G}{\partial \mu}(0, -1) 
= 1 + \int^{\frac{1}{2a}}_{\frac{3}{2a}} \frac{a^2}{2} \left(-a \left(s - \frac{3}{2a}\right)^2 - \left(s - \frac{3}{2a}\right)\right) \, ds= 1 + \frac{1}{6} - \frac{1}{4} \neq 0, 
\]
and, by the Implicit Function Theorem, the equation \( G(M, \mu) = 0 \) has a differentiable solution \( \mu = \mu(M) \) in a neighborhood of \( M = 0 \), with \( \mu(0) = -1 \).

Returning to the variables $M$ and $m$,  $m = \frac{M}{\mu},$ we conclude that $L(M, m) = m$ has a smooth solution
$m = \frac{M}{{\mu(M)} }=: \varphi(M)$
in a neighborhood of \( M = 0 \). Since
\[
\varphi'(0) = \lim_{M \to 0}{\varphi(M)}/{M} ={1}/{\mu(0)} = -1,
\]
and \( \hat{L} = \varphi^{-1} \) (the inverse), we conclude that \( \hat{L}'(0) = -1 \)
\subsection{Local analysis of \(M = \Sigma_{a}(M,m)\)}
It can be proven that the slope of the curve \(M = \Sigma_{a}(M,m)\) at \((0, 0)\) is equal to \(-1\) when $a=-\frac{37}{24}$. We present the calculations below.
We will  again express everything in terms of \(\mu= \mu(m) = M/m\).  If \(\mu\)  is fixed and \(m \to 0\) then \(M \to 0\) and, consequently, \(z(s) \to 0\) on \([-1, 1]\). Thus, for each $\mu$ close to $-1$ and $m \to 0$, we have:
\[
D_a:=D_{a}(M,m) \to 0, \quad  \beta_{-} \to \frac{1}{a} \left(\frac{1}{2} + \frac{1}{\mu} \right) \leq 1,
\quad \alpha_{-} \to \frac{1}{a} \left( -\frac{1}{2} + \frac{1}{\mu} \right) \leq 1, 
\]
\[
\beta_{+} \to \frac{1}{a} \left( 1 + \frac{\mu}{2} \right) \geq -1, \quad \alpha_{+} \to \frac{1}{a} \left( 1 - \frac{\mu}{2} \right) \geq -1,
\]
\[
\frac{a^2 M}{1+M}  =\frac{a^2 z(-1)}{1+z(-1)} \leq D_{a} = \frac{a^2 z(s_0-1)}{(1+z(s_0))^2(1+z(s_0-1))} \leq \frac{a^2 M}{(1+m)^2(1+M)}. 
\]
Therefore along the line \(M = \mu  m\), with $m \to 0$, we get \(D_{a}(M,m) / m \to  a^2 \mu\). 

Furthermore,
\[
\tilde{\alpha} \to \frac{1}{a} \left( 1 - \frac{1}{2 \mu} \right) \geq - 1, \quad
\tilde{\beta} \to \frac{1}{a} \left( 1 + \frac{1}{2 \mu} \right) \geq -1.
\]

Now, from \(M = \Sigma_{a}(M,m)\) we obtain that
\[
\frac{M}{m} = \int^0_{\tilde{\beta}} \frac{r(r(m)s)}{m} \, ds + \int_{\tilde{\alpha}}^{\tilde{\beta}} \frac{\frac{a}{m} \left( m + \frac{D_2}{2}(s-\tilde{\alpha})^2 \right)}{1 + m + \frac{D_2}{2}(s-\tilde{\alpha})^2} \, ds + \int^{\tilde{\alpha}}_{-1} \frac{a}{1+m} \, ds,
\]
so that, by letting $m \to 0$, we obtain
\[
\mu = \int_{\frac{1}{a}\left( 1 + \frac{1}{2\mu} \right) }^{0} a^2 s \, ds + \int_{\frac{1}{a} \left( 1 - \frac{1}{2\mu} \right)}^{\frac{1}{a} \left( 1 + \frac{1}{2\mu} \right) }  a \left( 1 + \frac{a^2\mu}{2}(s - \tilde{\alpha})^2 \right) \, ds
+a \left( \frac{1}{a}\left(1 - \frac{1}{2\mu}\right) + 1 \right)
\]
\[
= -\frac{1}{2} \left( 1 + \frac{1}{2\mu} \right)^2 + \frac{1}{\mu} + \frac{\mu}{6}\frac{1}{\mu^3} + a + 1 - \frac{1}{2\mu}
= a + \frac{1}{2} + \frac{1}{24 \mu^2}.
\]
 Setting $a=-{37}/{24}$, we obtain  that \ $0=24\mu^3+25\mu^2-1= (\mu+1)(24\mu^2 +\mu -1)$, from which we conclude that   \(\mu = -1\) is the only root close to $-1$.

\section{Appendix A: 
Scripts for $L_a(M,m)$ and Lemma \ref{L2.15} (a), (c), (d)}

\subsection*{\hspace{-8.9mm}}

\vspace{-5mm}

\begin{lstlisting}
function q = L(M,m) %this function compute L(m,M)
longprecision(32);
a=-37/24;
rM=M./(1+M);
rm=m/(1+m);
B=sqrt(-2*(1+M)*rm);
C=B./(a.*rm);
apl=(-24/37)*(1+M-0.5.*rM./rm);
bpl=(-24/37)*(1+M+0.5.*rM./rm);
if apl>= -1
I13= a*rM-1-0.5*rM*(1+rM)/rm+log(1+M+0.5*(rM)^2/rm)/rM;
I2=rM/rm +log((rM+B)/(-rM+B))/B;
q=(I2+I13);
else 
    if bpl>= -1
q= -1-M-0.5.*rM/rm+log(1+M+0.5*(rM)^2/rm)/rM +a.*(bpl+1) ...
-log(abs((apl+1-C)./(apl+1+C))*abs((bpl-apl-C)./(bpl-apl+C)))/B;
    else 
q= a+log(1-a.*rM)./rM;
    end
end
end 
\end{lstlisting}
\newpage 
\begin{lstlisting}
X(1)=intval(-0.61);
m(1)=inf(X(1));
Z(2)=L(-X(1),X(1));
m(2)=inf(Z(2));
X(2)=intval(m(2));
j=2;
while (m(j-1)<m(j)) & (m(j) <-0.009)
    j=j+1;
    Z(j)=L(-X(j-1),X(j-1));
    m(j)=inf(Z(j));
    X(j)=intval(m(j)); 
end
j
m(j-1)
m(j)
\end{lstlisting}
    
\begin{lstlisting}
X=intval(-0.009);
j=1;
function y=C(m)
Am= -(37/24)*m/(1+m)-1+(1+m)*log(1+m)/m; 
y=L(Am,m)
end
while (sup (X)>=-1/9) & (sup(C(X)) <inf(X))
    X=intval(sup(C(X)));
    j=j+1; 
end
sup(X)
j
sup(C(X))
\end{lstlisting}

\begin{lstlisting}
X=intval(-0.61);
j=1;
function y=C(m)
Am= -(37/24)*m/(1+m)-1+(1+m)*log(1+m)/m; 
y=L(Am,m);
end
while (inf(X)<-0.25) & (inf(C(X)) >sup(X))
    X=intval(inf(C(X)));
    j=j+1; 
end
inf(X)
j
\end{lstlisting}

\newpage 
\section{Appendix B: 
Lemma 3.16,  computation  of $\partial  L_{-37/24}(M,m)/\partial m$ and its estimates} 
\subsection*{\hspace{-8.9mm}}
\vspace{-5mm}

\begin{lstlisting}
function q = dL(M,m)
a=-37/24;
rM=M./(1+M);
rm=m/(1+m);
drm=1/(1+m)^2;
B=sqrt(-2*(1+M)*rm);
C=B./(a.*rm);
DB=-(1+M)/B;
apl=(-24/37)*(1+M-0.5.*rM./rm);
bpl=(-24/37)*(1+M+0.5.*rM./rm);
dbpl=(12/37)*rM*drm./(rm)^2;
dapl=-dbpl;
DC=(-24/37).*(DB*rm-B)*drm./(rm)^2;
if M>= -(37/24)*m/(1+m)-1+(1+m)*log(1+m)/m disp('out')
else 
    if apl>= -1
I13= (0.5*rM*(1+rM)/(rm)^2- (0.5*(rM)^2/(rm)^2)*(1+M+0.5*(rM)^2/rm)^(-1)/rM)*drm;
I2=(-rM/(rm)^2 -DB*log((rM+B)/(-rM+B))/B^2+ DB*((rM+B)^(-1)-(-rM+B)^(-1))/B)*drm;
q=(I2+I13);
    else 
        if bpl>= -1
        q = 0.5.*rM*drm/(rm)^2- (0.5*(rM)^2/(rm)^2)*drm*(1+M+0.5*(rM)^2/rm)^(-1)/rM +a.*dbpl  ...
        + log(abs((apl+1-C)./(apl+1+C))...
        *abs((bpl-apl-C)./(bpl-apl+C)))*DB*drm/B^2 ...
        - B^(-1)*((dapl-DC)/(apl+1-C)- (dapl+DC)/(apl+1+C) + (dbpl-dapl-DC)/(bpl-apl-C) ...
            - (dbpl-dapl+DC)/(bpl-apl+C));
        else  
        q= 0;
    end
end
end 
end
\end{lstlisting}
\begin{lstlisting}
for j=1:300 %here we consider the sub-interval $[-0.25, -0.13]$
m1(j)=-0.13-(j-1)*0.0004;
m2(j)=-0.13-j*0.0004;
a2(j)=-m1(j);
A(j)=-(37/24)*m2(j)/(1+m2(j))-1+(1+m2(j))*log(1+m2(j))/m2(j);
N(j)= floor((A(j)-a2(j))*10^3)+1;
h=10^(-3);
mY(j)= infsup(m2(j),m1(j));
for k=1:N(j)
W(j,k)= sup(dL(infsup(a2(j)+(k-1)*h,a2(j)+k*h), mY(j)));
end
end
max(max(W))
\end{lstlisting}

\section{Appendix D: MATLAB scripts for Lemma \ref{4.8}}

\subsection*{\hspace{-8.9mm}}
\vspace{-5mm}
\begin{lstlisting}
function z = LDa(M,m,a) %rigorous computation of a low bound for D_a(M,m)
tic
longprecision(32);
rM =a.*M./(1+M);
rm =a.*m./(1+m);
bp=(-0.5+(1+M).*m./M)./a;
ap=(0.5+(1+M).*m./M)./a;
am=0.5.*rM./(a.*rm)+M./rM;
bm=-0.5.*rM./(a.*rm)+M./rM;
bmm=max(-1,bm);
function z = zeta(s,M,m,a)
z=M.*(s <=bm)+ (0.5.*rm.*a.*(s-bm).^2+M).*(s>bm & s <=am)+ rM.*s.*(s >am & s <=ap) + (0.5.*rM.*a.*(s-bp).^2+m).*(s>ap & s <=bp)+m.*(s>bp);
end 
function z = frazeta(s,M,m,a)
z=a.^2*zeta(s-1,M,m,a)./((1+zeta(s,M,m,a)).^2.*(1+zeta(s-1,M,m,a)));
end
X= intval(bmm+1);
Y=intval(ap);
x1=inf(frazeta(X,M,m,a));
x2=inf(frazeta(Y,M,m,a));
Co=[-0.5.*a.^2*mid(rM).*(mid(rm)).^2 0 -2.*a.*mid(M).*mid(rm).*mid(rM) a.*mid(rm).*(1+mid(rM).*(1+mid(bm))) -2.*a.*mid(M).^2];
p=polynom(Co);
r=roots(Co);
for j=1:4
    if (isreal(r(j))>0) && (r(j) >=max(0,-1-bm)) && (r(j) <ap-1-bm) && (r(j) <am+1-1-bm) 
        X(j)=verifypoly(p,r(j))+1+bm;
        x3(j)=inf(frazeta(X(j),M,m,a));
    else 
        x3(j)=0; 
    end
end
if (ap-am >=1)
    T1=hull(am+1+eps,ap-eps);
    tak1=inf(frazeta(T1,M,m,a));
    T2=hull(ap+eps, min(bp,1));
    tak2=inf(frazeta(T1,M,m,a));
    x4=max(tak1,tak2);
end
if (ap-am <1)
    T3a=hull(ap+eps, ap+(am-ap+1)./8);
    T3b=hull(ap+(am-ap+1)./8, ap+(am-ap+1)./4);
    T3c=hull(ap+(am-ap+1)./4, ap+(am-ap+1)./2);
    T3d=hull(ap+(am-ap+1)./2, am+1-eps);
    tak3a=inf(frazeta(T3a,M,m,a));
    tak3b=inf(frazeta(T3b,M,m,a));
    tak3c=inf(frazeta(T3c,M,m,a));
    tak3d=inf(frazeta(T3d,M,m,a));
    T4a=hull(am+1+eps, am+1 + (min(bp,1)-am-1)./8);
    T4b=hull(am+1 + (min(bp,1)-am-1)./8, am+1 + (min(bp,1)-am-1)./4);
    T4c=hull(am+1 + (min(bp,1)-am-1)./4, am+1 + (min(bp,1)-am-1)./2);
    T4d=hull(am+1 + (min(bp,1)-am-1)./2, min(bp,1));
    tak4a=inf(frazeta(T4a,M,m,a)); 
    tak4b=inf(frazeta(T4b,M,m,a)); 
    tak4c=inf(frazeta(T4c,M,m,a)); 
    tak4d=inf(frazeta(T4d,M,m,a)); 
    x4=max([tak3a tak3b tak3c tak3d tak4a tak4b tak4c]);
end
v=[x1 x2 x3 x4];
z=max(v);
toc
end
    
\end{lstlisting}

\begin{lstlisting}
longprecision(32);
M(1)=0.0093-eps;
j=1;
z(1)=1+0.25.*a.^2./LDa(M(1),-M(1),a);
M(2)=1./(z(1)+sqrt((z(1)).^2-1))-eps;
while (M(j)<M(j+1)) & (M(j+1) <0.26)
j=j+1;
z(j)=1+0.25.*a.^2./LDa(M(j),-M(j),a);
M(j+1)=1./(z(j)+sqrt((z(j)).^2-1))-eps;
end 
j
M(j)
M(j+1)
\end{lstlisting}

\begin{lstlisting}
a=-3/2;
longprecision(32);
M(1)=0.21-eps;
j=1;
z(1)=1+0.25.*a.^2./LDa(0.27,-M(1),a);
M(2)=1./(z(1)+sqrt((z(1)).^2-1))-eps;
while (M(j)<M(j+1)) & (M(j+1) <0.26)
j=j+1;
z(j)=1+0.25.*a.^2./LDa(0.27,-M(j),a);
M(j+1)=1./(z(j)+sqrt((z(j)).^2-1))-eps;
end 
j
M(j)
M(j+1)

\end{lstlisting}

\section{Appendix E: MATLAB scripts for Lemma \ref{L4.18}}

\subsection*{\hspace{-8.9mm}}
\vspace{-5mm}
\begin{lstlisting}
function q = LSigma(M,m)
%this function compute a lower estimate for Sigma(m,M) from the paper
longprecision(32);
a=-37/24;
tic
rM=M./(1+M);
rm=m/(1+m);
Za=LDa(M,m,a)./a^2;
B=sqrt(2*(1+m)*Za);
ami=(-24/37)*(1+m-0.5.*rm./Za);
bmi=(-24/37)*(1+m+0.5.*rm./Za);
C=a.*sqrt(0.5.*Za./(1+m)).*(bmi+1);
if ami>= -1
I13= a*rm-1-0.5*rm.*(1+rm)./Za+log(1+m+0.5*(rm)^2/Za)/rm;
I2=rm/Za -2.*atan(rm./B)./B;
q=(I2+I13);
else 
    if bmi>= -1
q= -1-m-0.5.*rm/Za+log(1+m+0.5*(rm)^2/Za)/rm +a.*(bmi+1)- (2./Za).*atan(C); 
    else 
q= a+log(1-a.*rm)./rm;
    end
end
toc
end 
\end{lstlisting}
\begin{lstlisting}
MM(1)= intval(0.0093);
M(1)=inf(MM(1));
MM(2)=intval(LSigma(M(1),-M(1)));
M(2)=inf(MM(2));
j=1;
while (M(j) < M(j+1)) & (M(j)<0.22)
    j=j+1;
    MM(j+1)=intval(LSigma(M(j),-M(j)));
    M(j)=inf(MM(j));
end
j
M(j)
M(j-1)

\end{lstlisting}

\section{Appendix F: MATLAB scripts for Theorem  \ref{1.4}}

\subsection*{\hspace{-8.9mm}}
\vspace{-5mm}
\begin{lstlisting}
function z = UDa(M,m,a) 
%rigorous computation of an upper bound for D_a(M,m)
tic
longprecision(32);
rM =a.*M./(1+M);
rm =a.*m./(1+m);
bp=(-0.5+(1+M).*m./M)./a; %key points in the definition of \tilde z=:zet, bm < am <ap<bp  
ap=(0.5+(1+M).*m./M)./a;
am=0.5.*rM./(a.*rm)+M./rM;
bm=-0.5.*rM./(a.*rm)+M./rM;
bmm=max(-1,bm);
    function z = zet(s,M,m,a) %computation of \tilde z
    z=M.*(s <=bm)+ (0.5.*rm.*a.*(s-bm).^2+M).*(s>bm & s <=am)+ rM.*s.*(s >am & s <=ap) + (0.5.*rM.*a.*(s-bp).^2+m).*(s>ap & s <=bp)+m.*(s>bp);
    end 
    function z = frazet(s,M,m,a) %computation of the fraction to maximize
    z=a.^2*zet(s-1,M,m,a)./((1+zet(s,M,m,a)).^2.*(1+zet(s-1,M,m,a)));
    end
X= intval(bmm+1); %left terminal point
Y1=intval(ap); %one of connection points
Y2=intval(am+1); % other connection point
x1=sup(frazet(X,M,m,a)); % upper bounds for evaluations in the terminal points 
x2=sup(frazet(Y1,M,m,a));
x2a=sup(frazet(Y2,M,m,a));
%the next polynomial determines the critical points (if any) of frazet on the
%interval [0, \min(ap, am+1)]. If these points do not exists, we evaluate at Y
Co=[-0.5.*a.^2*mid(rM).*(mid(rm)).^2 0 -2.*a.*mid(M).*mid(rm).*mid(rM) a.*mid(rm).*(1+mid(rM).*(1+mid(bm))) -2.*a.*mid(M).^2];
p=polynom(Co);
r=roots(Co);
for j=1:4
    if (isreal(r(j))>0) && (r(j) >=max(0,-1-bm)) && (r(j) <ap-1-bm) && (r(j) <am+1-1-bm) 
        X(j)=verifypoly(p,r(j))+1+bm; %if there exists some approx. to real root r(j) in the interval, X(j) gives a verified interval around it 
        x3(j)=sup(frazet(X(j),M,m,a)); %upper bound for evaluations in these critical points 
    else 
        x3(j)=0; %if the critical points do not exist 
    end
end
% hence, max(x1,x2,x3) compute an upper bound for Da on [0, \min(ap,am+1)].
%Still, we need consider the case when ap >= am+1 and when ap < am+1 
if (ap-am >=1)
    T1=hull(am+1+eps,ap-eps);
    tak1=sup(frazet(T1,M,m,a));
    T2=hull(ap+eps, min(bp,1));
    tak2=sup(frazet(T1,M,m,a));
    x4=max(tak1,tak2);
end
if (ap-am <1)
    L=32; %first, we look for the upper bound on [ap,am+1] by using a direct interval evaluation 
    T3(1)=hull(ap+eps, ap+(am-ap+1)./L);
    tak3(1)=sup(frazet(T3(1),M,m,a));
    T3(L)=hull(ap+(L-1).*(am-ap+1)./L, am+1-eps);
    tak3(L)=sup(frazet(T3(L),M,m,a));
    for k=2:(L-1)
    T3(k)=hull(ap+(k-1).*(am-ap+1)./L, ap+k.*(am-ap+1)./L);
    tak3(k)=sup(frazet(T3(k),M,m,a));
    end
    H=10; %then we look for the upper bound on [am+1, min(bp,1)]
    T4(1)=hull(am+1+eps, am+1 + (min(bp,1)-am-1)./H); 
    tak4(1)=sup(frazet(T4(1),M,m,a));
    T4(H)=hull(am+1 + (H-1).*(min(bp,1)-am-1)./H, min(bp,1));
    tak4(H)=sup(frazet(T4(H),M,m,a));
    for i=2:(H-1)
    T4(i)=hull(am+1 + (i-1).*(min(bp,1)-am-1)./H, am+1 + i.*(min(bp,1)-am-1)./H);
    tak4(i)=sup(frazet(T4(1),M,m,a));
    end
    x4=max([tak3 tak4]);
end
v=[x1 x2 x2a x3 x4];
z=max(v);
toc
end

\end{lstlisting}

\begin{lstlisting}
function q = USigma(M,m)
%this function compute an upper estimate for Sigma(M,m) from the paper
%longinit('WithErrorTerm');
longprecision(32);
a=-37/24;
tic
rM=M./(1+M);
rm=m/(1+m);
Za=UDa(M,m,a)./a^2;
B=sqrt(2*(1+m)*Za);
ami=(-24/37)*(1+m-0.5.*rm./Za);
bmi=(-24/37)*(1+m+0.5.*rm./Za);
C=a.*sqrt(0.5.*Za./(1+m)).*(bmi+1);
if ami>= -1
I13= a*rm-1-0.5*rm.*(1+rm)./Za+log(1+m+0.5*(rm)^2/Za)/rm;
I2=rm/Za -2.*atan(rm./B)./B;
q=(I2+I13);
else 
    if bmi>= -1
q= -1-m-0.5.*rm/Za+log(1+m+0.5*(rm)^2/Za)/rm +a.*(bmi+1)- (2./Za).*atan(C); 
    else 
q= a+log(1-a.*rm)./rm;
    end
end
toc
end 
\end{lstlisting}
\begin{lstlisting}
function l = theta6(m) %rigorous computation of an upper bound for theta_6
longprecision(32);
a=-37/24;
Am=intval(-(37/24)*m./(1+m)-1+(1+m).*log(1+m)./m);
    l= sup(Am);
    for j=1:6
 l= intval(USigma(sup(l),m)); 
    end
end
\end{lstlisting}

\vspace{2mm}
\begin{lstlisting}
function z = m5(M) %this script compute a lower bound for m_5(M)
    l= intval(-M);
    r=intval(M);
    j=1;
    tic
    while (sup(l)<=inf(L(r,l))) & (j<=5)
 l= intval(inf(L(r,l))); 
 j=j+1;
    end
    z=inf(l)
    toc
end



\end{lstlisting}

\begin{lstlisting}
tic; %Iterative procedure, here for the interval [0.0094,0.0100002] 
tstart=tic;
M_0= intval(0.0100002); 
m_0= m5(M_0);
M_1=intval(sup(theta6(m_0)));
j=1; 
while (M_1 < M_0) & (M_1 > 0.0094)
j=j+1; 
M_0 =M_1;
m_0= m5(M_0);
M_1=intval(sup(theta6(m_0)));
end
j
M_0
m_0
telapsed = toc(tstart)
\end{lstlisting}

\bibliographystyle{siamplain}
\bibliography{references}
\end{document}